\documentclass[onefignum,onetabnum]{siamart171218}

\usepackage{lipsum}
\usepackage{amsfonts}
\usepackage{graphicx}
\usepackage{epstopdf}
\usepackage{algorithmic}
\usepackage{subcaption}
\ifpdf
  \DeclareGraphicsExtensions{.eps,.pdf,.png,.jpg}
\else
  \DeclareGraphicsExtensions{.eps}
\fi
\usepackage[nothm]{YK}

\newtheorem{remark}{Remark}

\newcommand{\normtwoinf}[1]{\|#1\|_{2,\infty}}
\newcommand{\normtwo}[1]{\left\|#1\right\|_{2}}
\newcommand{\normfro}[1]{\left\|#1\right\|_{F}}
\newcommand{\norminf}[1]{\left\|#1\right\|_{\infty}}
\newcommand{\nk}{n \times k}

\def\gap{{\sf gap}}
\def\hA{\widehat{A}}
\def\hV{\widehat{V}}
\def\hU{\widehat{U}}
\def\hX{\widehat{X}}
\def\hY{\widehat{Y}}
\def\orth{\mathbf{O}}
\def\reals{\mathbf{R}}
\def\sep{{\sf sep}}

\def\tU{\widetilde{U}}
\def\tV{\widetilde{V}}
\def\tW{\widetilde{W}}
\def\tE{\widetilde{E}}
\def\tY{\widetilde{Y}}
\def\tA{\widetilde{A}}
\def\tX{\widetilde{X}}
\def\tV{\widetilde{V}}

\newcommand{\emthin}{\thinspace\textemdash\thinspace}

\headers{Uniform bounds for invariant subspace perturbations}{A. Damle and Y. Sun}

\title{Uniform bounds for invariant subspace perturbations\thanks{Submitted to the editors \today.
\funding{This work was partially funded by the National Science Foundation under awards DMS-1830274 (AD) and DMS-1830247 (YS).}}}

\author{Anil Damle\thanks{Department of Computer Science, Cornell
    University, Ithaca, NY (\email{damle@cornell.edu}).}
\and Yuekai Sun\thanks{Department of Statistics, University of Michigan, Ann Arbor, MI
  (\email{yuekai@umich.edu}).}}

\usepackage{amsopn}

\ifpdf
\hypersetup{
  pdftitle={Uniform bounds for invariant subspace perturbations},
  pdfauthor={A. Damle and Y. Sun}
}
\fi

\begin{document}

\maketitle

\begin{abstract}
  For a fixed symmetric matrix $A$ and symmetric perturbation $E$ we develop purely deterministic bounds on how invariant subspaces of $A$ and $A+E$ can differ when measured by a suitable ``row-wise'' metric rather than via traditional measures of subspace distance. Understanding perturbations of invariant subspaces with respect to such metrics is becoming increasingly important across a wide variety of applications and therefore necessitates new theoretical developments. Under minimal assumptions we develop new bounds on subspace perturbations under the two-to-infinity matrix norm and show in what settings these row-wise differences in the invariant subspaces can be significantly smaller than the analogous two or Frobenius norm differences. We also demonstrate that the constitutive pieces of our bounds are necessary absent additional assumptions and, therefore, our results provide a natural starting point for further analysis of specific problems. Lastly, we briefly discuss extensions of our bounds to scenarios where $A$ and/or $E$ are non-normal matrices.
\end{abstract}



\section{Introduction}

Given a matrix $A,$ it is natural to try and ``understand'' how properties of $A$ change after it has been perturbed in some way. Understanding can take many forms, as can the way we choose to perturb $A.$ In this work we are primarily concerned with additive perturbations and understanding how spectral properties of the matrix change\textemdash \emph{i.e.,} given a perturbation $E$ how do certain invariant subspaces of $A+E$ relate to those of $A.$ This is a long standing question and one that has received extensive attention over the years. This includes the well-known Davis-Kahan theorem~\cite{davis1970rotation}, work by Wedin~\cite{Wedin1972}, and more general and extensive perturbation theory; see, \emph{e.g.,}~\cite{stewart1990Matrix} for an overview. 

Nevertheless, many problems of growing interest in mathematics, statistics, and computer science require new variants of such theory. Most notably, this manifests as modifications to the metrics we use to assess the similarity of invariant subspaces of $A$ and $A+E.$ Concretely, whereas traditional theory is often interested in classical notions of subspace distance measured by spectral or Frobenius norms, we will be interested in row-wise\footnote{As elaborated on later, we must somewhat carefully define what it means for two subspaces to be close row-wise.} (or closely related) measures of error. 
In Section~\ref{sec:prelim} we will formally outline the specifics of these metrics, contrast them to traditionally considered metrics, and provide additional preliminary material relevant to our work.

The impetus for these new types of bounds is often, though not exclusively, problems arising in statistics and computer science such as matrix completion~\cite{candes2009Exact,keshavan2010matrix}, principal component analysis~\cite{cai2013sparse,nadler2008finite}, robust factor analysis~\cite{fan2018Robust}, spectral clustering~\cite{abbe2017Entrywise,Athreya2017rdpg,damle2018simple,lei2015consistency,rohe2011spectral,vonluxburg2008Consistency}, and more. In these settings $A$ will often represent some model and a given instance of the model $\hA$ can be thought of as a (random) perturbation to this baseline, \emph{i.e.} $\hA=A+E.$ Many models $A$ have highly structured and meaningful invariant subspaces whose properties form the basis for a wide variety of algorithmic development and analysis of the underlying problem. Therefore, given $\hA$ we would like to understand if that structure can still be reliably leveraged. For many of these applications traditional measures of distance do not not necessarily provide adequate control over changes to the invariant subspaces.

A simple, concrete illustration of the types of bounds we will develop is encapsulated in the following scenario. Given a rank-$k$ symmetric matrix $A,$ an $\nk$ matrix $V$ with orthonormal columns representing the subspace associated with the non-zero eigenvalues, and a symmetric perturbation $E,$ when is the dominant invariant subspace of $A+E$ closer to $V$ row-wise than may be expected based on the smallest non-zero eigenvalue of $A$ and the spectral norm of $E$? 

More specifically, if we let $k=1$ this question reduces to understanding when
\[
\min_{s=\pm 1} \|\widehat{v} - sv\|_{\infty} \ll \min_{s=\pm 1} \|\widehat{v} - sv\|_2,
\]
where $v$ is the eigenvector of $A$ associated the eigenvalue $\lambda\neq 0$ and $\widehat{v}$ is the eigenvector of $A+E$ associated with the eigenvalue largest in magnitude (assuming $\|E\|_2$ is small enough for this to be the appropriate pairing). Our main results restricted to this setting answer such a question by providing explicit bounds on $\min_{s=\pm 1} \|\widehat{v} - sv\|_{\infty}.$ In particular we show that
\[
\min_{s=\pm 1} \|\widehat{v} - sv\|_{\infty} \leq 8\|v\|_{\infty}\frac{\|Ev\|^2_2}{\lambda^2} + 2\zeta\frac{\|Ev\|_{\infty}}{\lambda} + 4 \zeta \frac{\|Ev\|_2}{\lambda^2}\max_{i}\|E_{i,:}\|_2,
\]
where $\zeta = \|I-vv^T\|_{\infty}.$ Notably, this upper bound can be substantially smaller than a traditional Davis-Kahan bound on ${\|vv^T-\widehat{v}\widehat{v}^T\|_2}$ that, up to constants, behaves like $\|Ev\|_2/\lambda.$

As illustrated by the above bound and the more general results we present later, in situations where $E$ and $V$ have relatively uniform row norms (\emph{i.e.,} they are incoherent~\cite{candes2009Exact}) we may expect significantly better bounds than what is captured by traditional subspace perturbation theory. We will formalize these results for symmetric matrices in Section~\ref{sec:main} where we also provide proofs and investigate the behavior of our bounds. To complement our theoretical developments, Section~\ref{sec:numerics} provides several numerical examples illustrating our bounds in appropriate scenarios.

Given the potential usefulness of such bounds and the extent of relevant applications, this area has received significant attention over the past several years~\cite{abbe2017Entrywise,cape2017Twotoinfinity,eldridge2017Unperturbed,fan2018eigenvector}. Our main contributions are summarized as:
\begin{enumerate}
\item We develop new deterministic row-wise perturbation bounds for orthonormal bases of invariant subspaces of symmetric matrices. Prior work often entangles deterministic bounds with assumptions and/or analysis tailored to specific random settings. Nevertheless, our deterministic bounds are easily amenable to further analysis in the probabilistic settings as illustrated in Section~\ref{subsec:probSetting}. Furthermore, in Section~\ref{sec:nonnormal} we provide some basic extensions of our results to situations where $A$ and/or $E$ are non-normal matrices.
\item We show that our bounds are sharp by constructing adversarial perturbations that saturate the bounds. 
\item Our perturbation bounds apply under more general conditions than preceding results and we argue that our assumptions are minimal in certain respects by considering specific examples.
\end{enumerate}
While some aspects of our bounds are broadly in alignment with prior work, as noted above others are new, rely on less restrictive assumptions, and are more directly interpretable. We will draw specific comparisons to existing results parallel to the development of our bounds in Section~\ref{sec:main}.

\section{Preliminaries}
\label{sec:prelim}
\subsection{Matrices}
Let $A\in\reals^{n\times n}$ be a symmetric (not necessarily positive-definite) matrix. We arrange its eigenvalues in descending order 
\[
\lambda_1 \ge \dots \ge \lambda_n,
\]
and denote its eigen-decomposition as
\begin{equation}
\label{eqn:eigen-decomp}
A = V_1\Lambda_1 V_1^T + V_2\Lambda_2 V_2^T,
\end{equation}
where $\Lambda_1 =  \diag(\lambda_1,\ldots,\lambda_r)$ and $\Lambda_2 =  \diag(\lambda_{r+1},\ldots,\lambda_n)$ and $V_1\in\reals^{n\times r}$ and $V_2\in\reals^{n\times(n-r)}$ are matrices with orthonormal columns whose columns are the associated eigenvectors.\footnote{In the case of repeated eigenvalues any orthonormal basis for the associated eigenspace suffices.} To make this splitting well defined we assume that $\lambda_r > \lambda_{r+1}.$ In addition, for our later results it is important that $\Lambda_2$ explicitly include any zero eigenvalues of $A$ as they must be incorporated into our measure of how close $\Lambda_1$ and $\Lambda_2$ are. Because we are interested in algebraic orderings of the eigenvalues we use the term \emph{dominant} to refer to the algebraically largest eigenvalues (in contrast to the largest in magnitude). Lastly, for a matrix $A$ we let $\Lambda(A)$ denote the set of eigenvalues of $A.$

\begin{remark}
Note that the restriction to the $r$ algebraically largest eigenvalues is not essential. Our results will only depend on spectral separation (or closely related quantities) and therefore may be easily adapted to any isolated collection of $r$ eigenvalues. However, such an adaptation introduces notational overhead without adding anything fundamentally new. Similarly, with appropriate adaptation these results are applicable to magnitude based ordering of the eigenvalues\textemdash though the ordering itself may be more sensitive to perturbations than the associated subspaces. Therefore, to streamline the exposition we present everything for the $r$ algebraically largest eigenvalues.
\end{remark}

Now, let $\hA = A + E$ represent a perturbation of $A$ by a symmetric matrix $E$ and let $\hV_1\in\reals^{n\times r}$ be a matrix with orthonormal columns whose range is the $r$-dimensional invariant subspace of $\hA$ associated with the algebraically largest eigenvalues. For the moment we will assume $\hat{\lambda}_r > \hat{\lambda}_{r+1}$ so this notation is well defined, later assumptions we make will ensure this property. The primary contributions of this paper are centered around relating $V_1$ and $\hV_1.$ 

Throughout this work we will be interested in projections of matrices onto the invariant subspaces associated with $A$ (represented by $V_1$ and $V_2$). Therefore, for any matrix $B\in\reals^{n\times n}$ define $B_{i,j}$ as $V_i^TBV_j$ with $i,j\in\{1,2\}$. Lastly, for any matrix $B\in\reals^{n\times n}$ we define the Sylvester operator $S_B:\reals^{n\times r}\to\reals^{n\times r}$ as 
\[
S_B:Z\to ZB_{1,1} - B_{2,2}Z.
\] 
Note that we have embedded $V_1$ and $V_2$ directly into the definition of this Sylvester operator for convenience.

\subsection{Norms}
Throughout this paper, we let $\norm{\cdot}_1,$ $\norm{\cdot}_\infty,$ and $\normtwo{\cdot}$ denote the standard $\ell_p$ vector norms and their associated induced matrix norms. Similarly, we let $\normfro{\cdot}$ denote the Frobenius norm. In this work we will also be concerned with the two-to-infinity induced matrix norm. Specifically, we denote this norm by $\normtwoinf{\cdot}$ and note that for an $\nk$ matrix $B$ it can be defined as the maximum $\ell_2$ norm of rows of $B,$ \ie
\[
\normtwoinf{B} = \max_{i=1,\ldots,n}\normtwo{B_{i,:}}.  
\]

We outline a few easily verified properties of $\normtwoinf{\cdot}$ that will be useful later:
\begin{itemize}
\item[]  {\sc Unitary invariance from the right:} For any orthogonal $Z\in\reals^{k\times k}$
\[
\normtwoinf{BZ} = \normtwoinf{B}.
\]
In other words, the norm is invariant under orthogonal transforms from the right (though, notably, not from the left). This property follows immediately from the unitary invariance of the two-norm.
\item[]  {\sc Invariance to signed permutations from the left:} For any signed permutation $\Pi\in\reals^{n\times n}$
\[
\normtwoinf{\Pi B} = \normtwoinf{B}.
\]
This result follows from the invariance of $\|\cdot\|_{\infty}$ to signed permutations.
\item[] {\sc Sub-multiplicative relations:}  The relevant sub-multiplicative relationships are 
\[
{\normtwoinf{B_1B_2} \leq \normtwoinf{B_1}\normtwo{B_2}} \quad \text{and} \quad {\normtwoinf{B_1B_2}\leq\norminf{B_1}\normtwoinf{B_2}.}
\] 
These inequalities follow from the definition of induced matrix norms and consistency of $\|\cdot\|_2$ and $\|\cdot\|_{\infty}$.
\end{itemize}

\subsection{Subspace distances}
Given our fairly loose assumptions on the eigenvalues of $A,$ we cannot always talk about convergence to individual eigenvectors. Instead, we consider the distances between invariant subspaces. We will often (implicitly) associate matrices with orthonormal columns and subspaces, and refer to the range of a matrix $W$ as $\ran{W}.$ Given two $\nk$ matrices with orthonormal columns $W$ and $\tW$ the distance between the subspaces $\ran W$ and $\ran \tW$ is 
\[
\dist(W,\tW)\equiv\|WW^T-\tW\tW^T\|_2.
\]
Equivalent definitions include (see, \emph{e.g.,}~\cite[\S~2.6]{golub1996Matrix}):
\begin{itemize}
\item[] {\sc Complementary subspaces:} Let $W_2 \in\reals^{n \times n-k}$ be an orthonormal basis for the orthogonal complement of the subspace spanned by $W,$ then
\[
\dist(W,\tW)\equiv\|W_2^T\tW\|_2.
\]
\item[] {\sc Sine-$\Theta$ distance:} Let $\Theta(W,\tW)$ be a diagonal matrix containing the principle angles between $W$ and $\tW,$ then 
\[
\dist(W,\tW)\equiv\|\sin \Theta(W,\tW)\|_2.
\]
\end{itemize}

The bounds we develop in Section~\ref{sec:main} will focus on a slightly different measure between $W$ and $\tW.$ Specifically, we will be concerned with the \emph{row-wise} error metric
\begin{equation}
\label{eqn:2inf_dist}
\min\{\normtwoinf{\tW U - W}:U\in\orth^k\},
\end{equation}
where the minimization over orthogonal matrices ensures the metric is appropriate for subspaces (as opposed to relying on a specific choice of basis). Further motivation for this metric is encapsulated by the following proposition.
\begin{proposition}
\label{prop:2infvec}
Given two orthonormal bases $W$ and $\tW,$ for any $w\in\ran{W}$ there exists a $\widetilde{w}\in\ran{\tW}$ such that
\[
\|w-\widetilde{w}\|_{\infty}\leq \min\{\normtwoinf{\tW U - W}:U\in\orth^k\}.
\]
\end{proposition}
In contrast, the analogy to Proposition~\ref{prop:2infvec} for traditional subspace distances would provide a bound on $\|w-\widetilde{w}\|_2$.

Conceptually, the metric~\eqref{eqn:2inf_dist} is closely related to the so-called orthogonal Procrustes problem
\begin{equation}
\label{eqn:F_dist}
\min\{\|\tW U - W\|_F:U\in\orth^k\},
\end{equation}
which is well-studied and has a known solution easily computable via the SVD of $\tW^TW$ (see, \emph{e.g.,} \cite[\S~12.4]{golub1996Matrix}). In fact, the orthogonal Procrustes problem can also be related to subspace distance since (see, \emph{e.g.,}~\cite{cai2013sparse} and~\cite[\S~12.4]{golub1996Matrix})
\[
\|\sin \Theta(W,\tW)\|_2 \leq \min\{\|\tW U - W\|_F:U\in\orth^k\} \leq \sqrt{2k}\|\sin \Theta(W,\tW)\|_2.
\] 
Notably, the entry-wise definitions of $\normtwoinf{\cdot}$ and $\|\cdot\|_F$ immediately show it is plausible that~\eqref{eqn:2inf_dist} can be considerably (by a factor of $1/\sqrt{n}$) smaller than~\eqref{eqn:F_dist}. In short, using $\normtwoinf{\cdot}$ allows us to understand how well the error is distributed over the rows.  

\subsection{Applications of two-to-infinity distances}
From an application perspective there are many reasons why we may be interested in distances measured via $\normtwoinf{\cdot}.$ For example, spectral algorithms for clustering interpret invariant subspaces row-wise to build low-dimensional embeddings of nodes in a graph or points in a point cloud~\cite{von2007tutorial} (specifically, rows of $V_1$ correspond to an $r$-dimensional embedding of the graph nodes or data points). Therefore, concrete analysis of the performance of spectral algorithms on model problems or justification of observed performance on real-world data benefits from ``point-wise'' control over perturbations to the embedding. 

Concretely, related work has been used to show spectral methods are information theoretically optimal for the Stochastic Block Model (SBM) with two blocks in the regime where node degree grows logarithmically with graph size~\cite{abbe2017Entrywise}. (See~\cite{abbe2017Community} for a more general perspective on recent advances for this problem, and~\cite{le2018Concentration} for how random matrix theory and concentration results play a role in community detection.) Algorithmically, it is significantly more difficult to analyze situations with more than two blocks, but we believe our results paired with specific formulations of spectral algorithms~\cite{damle2018simple} may allow for more general analysis. Other problems that benefit from such bounds include so-called synchronization problems~\cite{perry2016optimality,singer2011angular}, recovery of unknown mixture components via spectral methods~\cite{vonluxburg2008Consistency}, and analysis of algorithmic performance on more general graph models such as random dot product graphs~\cite{Athreya2017rdpg}.

\subsection{Separation of matrices}
An important concept for our work is the separation of matrices in various norms. Specifically, for any two matrices $B\in\reals^{\ell \times \ell}$ and $C\in\reals^{m \times m}$ and norm $\|\cdot\|_*$ on $\reals^{m\times \ell}$, define\footnote{This specific form of $\sep$ differs slightly notationally, though not mathematically, from the standard way it is written for the two or Frobenius norm. Since we will ultimately be dealing with norms where $\|B\|_*\neq\|B^T\|_*$ this definition is required for consistency.}
\[
\sep_*(B,C) = \inf\{\|ZB-CZ\|_*:\|Z\|_* = 1\}.
\]
Perhaps the most pervasive use of $\sep$ is in traditional invariant subspace perturbation theory. In fact, for normal matrices $B$ and $C$
\[
\sep_F(B,C) = \min_{\lambda\in\Lambda(B),\mu\in\Lambda(C)} \lvert \lambda - \mu\rvert,
\]
thereby recovering the commonly used notion of an eigengap (see, \emph{e.g.,} \cite{stewart1990Matrix}).

Importantly, and in alignment with our algebraic ordering of eigenvalues above, $\sep$ is shift invariant in any norm, \emph{i.e.} 
\[
\sep_*(B+\xi I,C + \xi I) = \sep_*(B,C)
\]
for any $\xi\in\reals.$ Furthermore, in any unitarily invariant norm $\sep$ is relatively insensitive to perturbations of small spectral norm. In other words (see Proposition 2.1 of \cite{karow2014Perturbation})
\[
\sep_2(B+E_B,C + E_C) \geq \sep_2(B,C) - \|E_B\|_2 - \|E_C\|_2
\]
and
\[
\sep_F(B+E_B,C + E_C) \geq \sep_F(B,C) - \|E_B\|_2 - \|E_C\|_2
\]
for $E_B\in\reals^{\ell\times \ell}$ and $E_C\in\reals^{m\times m}.$ Lastly, given diagonal matrices it is possible to control $\sep$ in a variety of norms necessary for our work. These bounds are captured collectively in Lemma~\ref{lem:sep_diag} (the results about $\sep_2$ and $\sep_F$ are well known).
\begin{lemma}
\label{lem:sep_diag}
Let $D_1\in\reals^{\ell \times \ell}$ and $D_2\in\reals^{m \times m}$ be  diagonal matrices and assume that $\lambda_{\min}(D_1) \geq \lambda_{\max}(D_2),$ then 
\begin{equation}
\sep_2(D_1,D_2) = \sep_F(D_1,D_2) = \sep_{2,\infty}(D_1,D_2) = \lambda_{\min}(D_1) - \lambda_{\max}(D_2).
\end{equation} 
\end{lemma}
\begin{proof}
We defer the proof to Appendix~\ref{sec:proofSepDiag}. 
\end{proof}

In addition to the above ``canonical'' definition of separation, some of our bounds requires that we introduce a slight variant of $\sep.$ In particular, we will occasionally consider the separation measured only over matrices in a linear subspace. More specifically, let $W\in\reals^{n\times k}$ be an orthonormal basis for a $k$-dimensional linear subspace and define
\[
\sep_{*,W}(B,C) = \inf\{\|ZB-CZ\|_*:Z\in \ran W, \|Z\|_* = 1\}.
\]
It is immediate that $\sep_{*,W}(B,C) \geq \sep_{*}(B,C)$ for any $W$ and therefore, as will become evident, consideration of this restricted version of $\sep$ can only improve our bounds. For us, the key use of this restricted separation quantity will be when $C = WD_2W^T$ for some diagonal matrix $D_2.$\footnote{One consequence of this restriction is that it will allow us to eliminate any artificial requirement that $\Lambda_1$ be separated from zero when $A$ is not low-rank.} 

In anticipation of its use later, we prove some results about this restricted version of $\sep$ analogous to our earlier statements. First, we generalize the notion of $\sep$ being shift invariant in Lemma~\ref{lem:sepU_diag}.
\begin{lemma}
\label{lem:sepU_diag}
Consider $B\in\reals^{\ell \times \ell}$ and $C\in\reals^{m \times m},$ and let $W\in\reals^{n\times m}$ with $n\geq m$ be a matrix with orthonormal columns. Then,
\[
\sep_{*,W}(B+\xi I, WCW^T+\xi WW^T) = \sep_{*,W}(B, WCW^T)
\]
for any $\xi\in\reals.$
\end{lemma}
\begin{proof}
The proof follows immediately from the fact that for any $Z\in\ran W$ $\xi WW^TZ = \xi Z$.
\end{proof}
Perhaps more importantly, and as illustrated in Lemma~\ref{lem:sepU_orth} in any unitarily invariant norm we can relate this restricted version of $\sep$ directly to $\sep(B,C).$   
\begin{lemma}
\label{lem:sepU_orth}
Consider $B\in\reals^{\ell \times \ell}$ and $C\in\reals^{m \times m},$ and let $W\in\reals^{n\times m}$ with $n\geq m$ be a matrix with orthonormal columns. Then, for any unitarily invariant norm $\|\cdot\|_*$
\[
\sep_{*,W}(B, WCW^T) = \sep_{*}(B,C)
\]
for any $\xi\in\reals.$
\end{lemma}
\begin{proof}
We first rewrite the infimum over $\ran W$ in terms of coefficients $X$ of $Z$ in the orthonormal basis $W$ as
\[
\sep_{*,W}(B,WCW^T) = \inf\{\|WXB-WCW^TWX\|_*:\|X\|_* = 1\},
\]
where we have used the fact that $\|\cdot\|_*$ is unitarily invariant to say that $\|WX\|_*=\|X\|_*.$ Using $W^TW=I$ and the unitary invariance of $\|\cdot\|_*$ once more concludes the proof.
\end{proof}

We now use these results to control a restricted version of $\sep_{(2,\infty)}$ (something that will be of particular importance to us) in terms of more traditional and directly interpretable quantities. These results are encapsulated in Lemma~\ref{lem:sepU_lowerbound} and we stress that they are worse case bounds that may be far from achieved in practice or provably loose in specific cases.\footnote{Concrete examples being when $C = 0$ or $W=I$ and $B$ and $C$ are diagonal (see Lemma~\ref{lem:sep_diag}) in which case $\sep_F = \sep_2 = \sep_{2,\infty}.$} Nevertheless, the fact that the $2,\infty$ norm is not unitarily invariant from the left significantly changes the landscape of possible outcomes.
\begin{lemma}
\label{lem:sepU_lowerbound}
Consider $B\in\reals^{\ell \times \ell}$ and $C\in\reals^{m \times m},$ and let $W\in\reals^{n\times m}$ with $n\geq m$ be a matrix with orthonormal columns. Then,
\[
\sep_{(2,\infty),W}(B, WCW^T) \geq \max \left\{\frac{1}{\sqrt{n}},\beta_W\right\}\sep_F(B,C),
\]
where
\[
\beta_{W} = \inf\left\{\frac{\normtwoinf{WX}}{\normtwoinf{W}}:\|X\|_F= 1\right\}.
\]
\end{lemma}
\begin{proof}
We prove two lower bounds on $\sep_{(2,\infty),W}(B, WCW^T)$ that always hold and then maximize over them. First, observe that for any $Z\in \ran W$ there exists an $X$ such that $Z = WX,$ hence 
\begin{align*}
\frac{\normtwoinf{ZB-WCW^TZ}}{\normtwoinf{Z}} &= \frac{\normtwoinf{WXB-WCX}}{\normtwoinf{WX}} \\
&\geq \frac{\|W(XB-CX)\|_F}{\sqrt{n}\|X\|_F}\\
&\geq \frac{\sep_F(B,C)}{\sqrt{n}},
\end{align*}
where we have used that $\frac{1}{\sqrt{n}}\|A\|_F \leq \normtwoinf{A} \leq \|A\|_F$ for any $A\in\reals^{n\times m}.$ 

Once again using that there exists an $X$ such that $Z = WX$ we see that 
\begin{align*}
\frac{\normtwoinf{ZB-WCW^TZ}}{\normtwoinf{Z}} &= \frac{\normtwoinf{WXB-WCX}}{\normtwoinf{WX}} \\
&\geq \frac{\normtwoinf{W(XB-CX)}}{\normtwoinf{W}\|X\|_F}\\
&\geq \frac{\normtwoinf{WT}}{\normtwoinf{W}},
\end{align*}
where $T = (XB-CX) / \|X\|_F.$ Since $\|T\|_F\geq \sep_F(B,C)$ we take an infimum of the lower bound over $X$ to conclude the proof.
\end{proof}

Lastly, in Lemma~\ref{lem:sepU_normlowerbound} we provide a direct bound on $\sep_{2,\infty}$ in terms of traditional matrix norms. In some situations, this may provide the most direct control over $\sep_{2,\infty}$ while in others it may be vacuous and one must resort to Lemma~\ref{lem:sepU_lowerbound} instead.
\begin{lemma}
\label{lem:sepU_normlowerbound}
Consider $B\in\reals^{\ell \times \ell}$ and $C\in\reals^{m \times m},$ and let $W\in\reals^{n\times m}$ with $n\geq m$ be a matrix with orthonormal columns. Then,
\[
\sep_{(2,\infty),W}(B, WCW^T) \geq \sigma_{\min}(B) - \|WCW^T\|_{\infty},
\]
where $\sigma_{\min}(B)$ is the minimal singular value of $B.$
\end{lemma}
\begin{proof}
Staring from $\sep_{(2,\infty),W}(B, WCW^T) \geq \sep_{(2,\infty)}(B, WCW^T)$ we have that
\begin{align*}
\frac{\normtwoinf{ZB-WCW^TZ}}{\normtwoinf{Z}} &\geq \frac{\normtwoinf{ZB} - \normtwoinf{WCW^TZ} }{\normtwoinf{Z}}\\
&\geq \frac{\sigma_{\min}\normtwoinf{Z} - \|WCW^T\|_{\infty}\normtwoinf{Z} }{\normtwoinf{Z}}\\
\geq \sigma_{\min}(B) - \|WCW^T\|_{\infty}.
\end{align*}
\end{proof}

\section{Main result}
\label{sec:main}
Given the notation and concepts from Section~\ref{sec:prelim} we may now proceed to present our core results bounding $\normtwoinf{\cdot}$ changes in invariant subspaces of symmetric matrices $A$ under symmetric perturbation; the proofs appear in Section~\ref{subsec:proof}. Notably, our main result includes bounds for a specific $U\in\orth^r,$ not just the minimum over all orthogonal matrices.

\begin{theorem}
\label{thm:main}
Let $A\in\reals^{n\times n}$ be symmetric with eigen-decomposition 
\[
A = V_1\Lambda_1V_1^T +V_2\Lambda_2V_2^T
\]
following the conventions of~\eqref{eqn:eigen-decomp} and
\[
\gap = \min\{\sep_2(\Lambda_1,\Lambda_2),\sep_{(2,\infty),V_2}(\Lambda_1,V_2\Lambda_2V_2^T)\}.
\] 
If $\|E\|_2 \le \frac{\gap}{5}$ then
\begin{align}
\normtwoinf{\hV_1\tU - V_1} &\le 8\normtwoinf{V_1}\left(\frac{\|E_{2,1}\|_2}{\sep_2(\Lambda_1,\Lambda_2)}\right)^2\\&\phantom{\le}  + 2\frac{\normtwoinf{V_2E_{2,1}}}{\gap}+ 4\frac{\normtwoinf{V_2E_{2,2}V_2^T}\|E_{2,1}\|_2}{\gap\times \sep_2(\Lambda_1,\Lambda_2)}, \nonumber
\label{eq:mainP}
\end{align}
where $\hV_1$ is any matrix with orthonormal columns whose range is the dominant $r$-dimensional invariant subspace of $\hA,$ and $\tU$ solves the orthogonal Procrustes problem
\begin{equation*}
\min\{\|\hV_1U - V_1\|_F:U\in\orth^r\}.
\end{equation*}
\end{theorem}

\begin{corollary}
\label{cor:main}
Following the notation of Theorem~\ref{thm:main} and under the same assumptions, we have that 
\begin{align}
\min\{\normtwoinf{\hV_1U - V_1}:U\in\orth^r\} &\le 8\normtwoinf{V_1}\left(\frac{\|E_{2,1}\|_2}{\sep_2(\Lambda_1,\Lambda_2)}\right)^2\\&\phantom{\le}+ 2\frac{\normtwoinf{V_2E_{2,1}}}{\gap} + 4\frac{\normtwoinf{V_2E_{2,2}V_2^T}\|E_{2,1}\|_2}{\gap\times \sep_2(\Lambda_1,\Lambda_2)}, \nonumber
\label{eq:main}
\end{align}
where $\hV_1$ is any matrix with orthonormal columns whose range is the dominant $r$-dimensional invariant subspace of $\hA.$
\end{corollary}

First, we briefly remark on the assumptions and implications of Theorem~\ref{thm:main}. The condition $\|E\|_2 \le \frac{\gap}{5}$ is standard in the literature; it ensures the two parts of $\Lambda(\hA)$ corresponding to the $r$-largest eigenvalues and the $n-r$ smallest eigenvalues of $A$ are disjoint.\footnote{Technically, the constant in the denominator just needs to be bigger than 4.} The first term on the right hand side is also expected, it looks like a traditional Davis-Kahan bound reduced by of the incoherence of $V_1$ and an additional factor of $\|E\|_2.$ The second term captures how (in)coherent $V_2V_2^TEV_1$ is, a term we often expect to be well controlled. Lastly, the third term is controlled by the incoherence of $E$ itself (relative to its spectral norm).\footnote{Note that if we define $\mu=\sqrt{n}\normtwoinf{V_1}$ it is possible to further simplify the bound by observing that $\normtwoinf{V_2V_2^TE} \leq \|V_2V_2^T\|_{\infty}\normtwoinf{E} \leq (1+\mu^2)\normtwoinf{E}.$} 

As we will see later, it is often the case that either the second or third term dominates the upper bound. In fact, in Section~\ref{sec:observations} we provide an illustrative example showing that both terms are necessary as part of our bound and are effectively tight. Furthermore, many random models for $E$ have the property that $\normtwoinf{V_2E_{2,2}V_2^T}$ and $\|E_{2,1}\|_2$ are on the same order, so in these settings using Theorem~\ref{thm:main} may only marginally improve on the classical Davis-Kahan bound. However, the third term can be more sharply controlled when $E$ is drawn from suitable random models by modifying the proof \emthin Theorem~\ref{thm:probG} explicitly shows how such an improvement is constructed. Lastly, in Section~\ref{sec:observations} we will argue that the presence of $\sep_{(2,\infty),V_2}(\Lambda_1,V_2\Lambda_2V_2^T)$ is essential, though Lemma~\ref{lem:sepU_lowerbound} provides some control over it via more interpretable quantities.

\begin{remark}
Of particular note, when $A$ is rank $r$ and, therefore, $\Lambda_2 = 0$ Theorem~\ref{thm:main} simplifies significantly since 
\[
\gap = \sep_2(\Lambda_1,\Lambda_2) =\sep_{(2,\infty),V_2}(\Lambda_1,V_2\Lambda_2V_2^T) = \min_{i=1,\ldots,r}\lvert \lambda_i \rvert.
\]
\end{remark} 

\begin{remark}
Eq.~\eqref{eq:mainP} of Theorem~\ref{thm:main} is a particularly useful result since there are circumstances where it is possible to estimate $\tU$ given only $\hV_1$ and some structural assumptions about $V_1.$ Algorithms based around this paradigm have been developed for spectral clustering~\cite{damle2018simple} and localization of basis functions in Kohn-Sham Density Functional Theory~\cite{damle2015compressed,damle2017computing}. 
\end{remark}

Prior to embarking on a proof of the main result, we present a corollary of independent interest. Corollary~\ref{cor:infBound} simplifies our result in the case where the infinity norm of $E$ is sufficiently bounded relative to the spectral gap and the incoherence of $V_1.$ This assumption allows us to remove the third term in the upper bound of Theorem~\ref{thm:main}.

\begin{corollary}
\label{cor:infBound}
Let $A\in\reals^{n\times n}$ be symmetric with eigen-decomposition 
\[
A = V_1\Lambda_1V_1^T +V_2\Lambda_2V_2^T
\] 
following the conventions of~\eqref{eqn:eigen-decomp},
\[
\gap = \min\{\sep_2(\Lambda_1,\Lambda_2),\sep_{(2,\infty),V_2}(\Lambda_1,V_2\Lambda_2V_2^T)\},\]
and $\mu = \sqrt{n}\normtwoinf{V_1}$. If $\|E\|_2 \le \frac{\gap}{5}$ and $\|E\|_{\infty}\leq \gap / (4+4\mu^2)$ then
\begin{equation*}
\normtwoinf{\hV_1\tU - V_1} \le 8\normtwoinf{V_1}\left(\frac{\|E_{2,1}\|_2}{\sep_2(\Lambda_1,\Lambda_2)}\right)^2 + 4\frac{\normtwoinf{V_2E_{2,1}}}{\gap},
\end{equation*}
where $\hV_1$ is any matrix with orthonormal columns whose range is the dominant $r$-dimensional invariant subspace of $\hA$ and $\tU$ solves the orthogonal Procrustes problem
\begin{equation*}
\min\{\|\hV_1U - V_1\|_F:U\in\orth^r\}.
\end{equation*}
\end{corollary}

\subsection{Related work}
The most closely related results to our own are the two-to-infinity bounds in \cite{cape2017Twotoinfinity}, though other results exist for single eigenvectors~\cite{eldridge2017Unperturbed} and for similar, though distinct, measures of subspace perturbations~\cite{abbe2017Entrywise}. The results in~\cite{cape2017Twotoinfinity} concern orthonormal bases of the singular subspaces of (possibly non-symmetric) matrices. However, when specialized to orthonormal bases of the invariant subspaces of symmetric matrices our results lead to sharper bounds. Specifically, in \cite{cape2017Twotoinfinity} the authors establish a general decomposition of $\hV_1\tU - V_1$, where $\tU$ solves
\[
\min\{\|\hV_1U - V_1\|_F:U\in\orth^r\},
\]
and deduce general bounds on $\normtwoinf{\hV_1\tU - V_1}$ through repeated use of the triangle inequality. 

More concretely, the first bound in~\cite{cape2017Twotoinfinity} (Theorem 3.7) is most similar to our main result, albeit proved in a significantly different manner. For symmetric positive definite matrices, combining Theorems 3.7 and 6.9 in~\cite{cape2017Twotoinfinity} (and recasting them using our notation) yields
\begin{align}
\label{eqn:cape2017}
\normtwoinf{\hV_1\tU - V_1} &\le 4 \normtwoinf{V_1}\left(\frac{\|E\|_2}{\sep_2(\Lambda_1,\Lambda_2)}\right)^2 + 2 \frac{\normtwoinf{V_2E_{2,1}}}{\lambda_r}\\&\phantom{\le} + 4 \frac{\normtwoinf{V_2E_{2,2}V_2^T}\|E\|_2}{\lambda_r \sep_2(\Lambda_1,\Lambda_2)} + 4 \frac{\normtwoinf{V_2\Lambda_2 V_2^T}\|E\|_2}{\lambda_r \sep_2(\Lambda_1,\Lambda_2)}.\nonumber
\end{align}

In the case where $A$ has rank $r$~\eqref{eqn:cape2017} is comparable to our main result.\footnote{While we are able to use the potentially smaller quantity $\|E_{2,1}\|_2$ in place of $\|E\|_2,$ for many random models on $E$ these two quantities behave essentially the same.} However, if $A$ is not low-rank our result implies tighter upper bounds. In this case, $\normtwoinf{V_2\Lambda_2V_2^T}$ is non-zero, and the right hand side of the bound from~\cite{cape2017Twotoinfinity} in~\eqref{eqn:cape2017} is dominated by a term depending on $\|E\|_2$\emthin the same behavior immediately implied by Davis-Kahan. If, for example, $E\in\reals^{n\times n}$ is a matrix of \iid\ $\mathcal{N}(0,1/n^2)$ random scalars, their bound implies $\normtwoinf{\hV_1\tU - V_1}$ vanishes at the rate $\tilde{\mathcal{O}}\left(1/\sqrt{n}\right)$. On the other hand, our bound shows that $\normtwoinf{\hV_1\tU - V_1}$ vanishes at the faster rate $\tilde{\mathcal{O}}(1/n)$\textemdash an observation illustrated in Section~\ref{sec:numerics}.

In \cite{abbe2017Entrywise} the authors develop similar results to Theorem~\ref{thm:main} as corollaries to their main results. Specifically, their final expressions in Theorem 2.1 and Corollary 2.1 are row-wise perturbations bounds on orthonormal bases of invariant subspaces. However, this is not the main focus of their work, and these results are generally looser than our bounds. Furthermore, in contrast to our results their bounds are not purely deterministic; they rely on probabilistic assumptions on the error matrix $E$ and additional assumptions on $A$ itself.

\subsection{Proof of the main result} 
\label{subsec:proof}
At a high level, our proof has two parts. In the first part, we develop a specific characterization of $\hV_1$ parametrized by a matrix $\hX$ that is a root of a quadratic matrix equation. In the second part, we show that $\normtwoinf{\hV_1 - V_1}$ is small under our stated assumptions. Throughout the proof we extensively leverage notation from Section~\ref{sec:prelim} to refer to projections of arbitrary matrices $B$ with respect to the representations $V_1$ and $V_2$ of invariant subspaces associated with $A$\emthin recall that for any $B\in\reals^{n\times n}$ $B_{i,j} = V_i^TBV_j.$

{\it Part 1:} Our starting point is the bound
\[
\min\{\|\hV_1U - V_1\|_{2,\infty}:U\in\orth^r\} \le \|\hV_1\tU - V_1\|_{2,\infty},
\]
where $\tU$ is the solution of the orthogonal Procrustes problem:
\begin{equation}
\min\{\|\hV_1U - V_1\|_F:U\in\orth^r\}.
\label{eq:orthogonalProcrustesProblem}
\end{equation}
Notably, the solution to this problem is well known and computable given $\hV_1$ and $V_1,$ which will prove useful in our numerical experiments. More pertinent to our needs at the moment, $\hV_1\tU$ is the closest matrix with orthonormal columns to $V_1$ in Frobenius norm whose range is the dominant $r$-dimensional invariant subspace of $\hA$.

We start by constructing a matrix with orthonormal columns $\hV_1$ whose range is the dominant $r$-dimensional invariant subspace of $\hA$ and a matrix $\hV_2$ characterizing the orthogonal complement of $\hV_1$. We will pick $\hV_1$ such that the solution to~\eqref{eq:orthogonalProcrustesProblem} is the identity, thereby simplifying the remainder of the proof. Nevertheless, any bound for this specific choice of $\hV_1$ simultaneously holds for any orthonormal basis of $\ran \hV_1$ since the discrepancy may be formally resolved by solving the orthogonal Procrustes problem. Specifically, consider
\begin{equation}
\hV_1 = (V_1 + V_2\hX)(I_r + \hX^T\hX)^{-\frac12},
\label{eq:hV1}
\end{equation}
and
\begin{equation}
\hV_2 = (V_2 - V_1\hX^T)(I_{(n-r)} + \hX\hX^T)^{-\frac12}
\label{eq:hV2}
\end{equation}
for some $\hX\in\reals^{(n-r)\times r}$. 

\begin{remark}
A clean construction of this characterization is to start with the general formula for an arbitrary invariant subspace $\hV_1 = V_1H + V_2X$ for some $H\in\reals^{r\times r}$ and $X\in\reals^{(n-r)\times r}.$ Requiring that $\hV_1^T\hV_1 = I$ ensures that $H$ is non-singular as long as $\|X\|_2 < 1,$ which is guaranteed by our assumptions. Letting $\hX = XH^{-1}$ the condition that $\hV_1$ has orthonormal columns shows that
\[
H^2 + H\hX^T\hX H = I.
\] 
Multiplying on the left and right by $H^{-1}$ we conclude that $H^{-2} = I+\hX^T\hX$ and arrive at~\eqref{eq:hV1}.
\end{remark}

It is not hard to check that $\hV_1$ and $\hV_2$ have orthonormal columns and their ranges are complementary subspaces of $\reals^n$. Thus $\ran\hV_1$ is an invariant subspace of $\hA$ if and only if 
\begin{equation}
\begin{aligned}
0 &= \hV_2^T\hA\hV_1 = -\hA_{2,1} + \hX\hA_{1,1} - \hA_{2,2}\hX + \hX\hA_{1,2}\hX.
\end{aligned}
\label{eq:hX}
\end{equation}
In other words, $\hX$ is a root of the map $F:\reals^{(n-r)\times r}\to\reals^{(n-r)\times r}$ defined as 
\[
\begin{aligned}
F:X\to -\hA_{2,1} + X\hA_{1,1} - \hA_{2,2}X + X\hA_{1,2}X.
\end{aligned}
\]
We find a root of $F$ by appealing to a Newton-type method (for root-finding). Starting at $X_0 = 0$, we construct the sequence
\begin{equation}
X_{t+1} \gets X_t - S_{\hA}^{-1}(F(X_t)).
\label{eq:NewtonIterationX}
\end{equation}
To characterize the limit of $(X_t)$ we appeal to the Newton-Kantorovich theorem (Theorem \ref{thm:NewtonKantorovich}). We remark that this construction is similar, but not identical to, that in \cite[\S 3]{stewart1973Error}.

\begin{lemma}
\label{lem:NewtonKantorovichX}
As long as $\|E\|_2 \le \frac{\sep_2(\Lambda_1,\Lambda_2)}{4}$, $(X_t)$ converges to $\hX$ such that $\hX$ satisfies \eqref{eq:hX} and $\|\hX\|_2 \le \frac{4\|E_{2,1}\|_2}{\sep_2(\Lambda_1,\Lambda_2)}$.
\end{lemma}
\begin{proof}
We defer the proof to Appendix~\ref{secA:NKX}.
\end{proof}

Since $\hX$ satisfies \eqref{eq:hX}, $\ran\hV_1$ is an invariant subspace of $\hA$. It remains to show that $\ran\hV_1$ is the dominant $r$-dimensional invariant subspace of $\hA$. We block-diagonalize $\hA$ to obtain

\begin{equation}
\label{eqn:AhatDiag}
\begin{bmatrix}\hV_1\mid\hV_2\end{bmatrix}^T\hA\begin{bmatrix}\hV_1\mid\hV_2\end{bmatrix} = \begin{bmatrix}\hV_1^T\hA\hV_1 & 0 \\ 0 & \hV_2^T\hA\hV_2\end{bmatrix}.
\end{equation}
The first diagonal block is
\[
\begin{aligned}
\hV_1^T\hA\hV_1 &= (I_r + \hX^T\hX)^{-\frac12}(V_1 + V_2\hX)^T\hA(V_1 + V_2\hX)(I_r + \hX^T\hX)^{-\frac12} \\
&= (I_r + \hX^T\hX)^{-\frac12}(\hA_{1,1} + \hA_{1,2}\hX + \hX^T\hA_{2,1} + \hX^T\hA_{2,2}\hX)(I_r + \hX^T\hX)^{-\frac12}.
\end{aligned}
\]
Recalling $\hX$ satisfies \eqref{eq:hX}, we have
\[
\hX^T\hA_{2,1} + \hX^T\hA_{2,2}\hX = \hX^T\hX\hA_{1,1} + \hX^T\hX\hA_{1,2}\hX.
\]
Plugging this expression into the right side of the preceding display, we obtain
\[
\begin{aligned}
\hV_1^T\hA\hV_1 &= (I + \hX^T\hX)^{-\frac12}(\hA_{1,1} + \hA_{1,2}\hX + \hX^T\hX\hA_{1,1} + \hX^T\hX\hA_{1,2}\hX)(I + \hX^T\hX)^{-\frac12} \\
&= (I + \hX^T\hX)^{-\frac12}(I + \hX^T\hX)(\hA_{1,1} + \hA_{1,2}\hX)(I + \hX^T\hX)^{-\frac12} \\
&= (I + \hX^T\hX)^{\frac12}(\hA_{1,1} + \hA_{1,2}\hX)(I + \hX^T\hX)^{-\frac12}.
\end{aligned}
\]
In other words, the first diagonal block is similar to $\hA_{1,1} + \hA_{1,2}\hX$. This implies that
\[
\begin{aligned}
\Lambda(\hV_1^T\hA\hV_1) &= \Lambda(\hA_{1,1} + \hA_{1,2}\hX) \\
&= \Lambda(\Lambda_1 + E_{1,1} + E_{1,2}\hX) \\
&\subset\Lambda(\Lambda_1) + (\|E_{1,1}\|_2 + \|E_{1,2}\|_2\|\hX\|_2)[-1,1] & \text{(Bauer-Fike theorem)} \\
&\subset\Lambda(\Lambda_1) + ({\textstyle\|E\|_2 + \frac{4\|E\|_2^2}{\sep_2(\Lambda_1,\Lambda_2)}})[-1,1] &\textstyle \text{($\|\hX\|_2 \le \frac{4\|E\|_2}{\sep_2(\Lambda_1,\Lambda_2)}$)} \\
&\textstyle\subset\Lambda(\Lambda_1) + 2\|E\|_2[-1,1] &\textstyle \text{($\frac{\|E\|_2}{\sep_2(\Lambda_2,\Lambda_2)} \le \frac14$),}
\end{aligned}
\]
where, as before, $\Lambda(A)$ is defined to be the set of eigenvalues of the matrix $A.$

Similarly, it is possible to show that the second diagonal block is similar to $\hA_{2,2} - \hX\hA_{1,2}$ and
\[\textstyle
\Lambda(\hV_2^T\hA\hV_2) \subset\Lambda(\Lambda_2) + 2\|E\|_2[-1,1].
\]
Recalling $\|E\|_2 \le \frac{\sep_2(\Lambda_1,\Lambda_2)}{5}$, we have
\[
\begin{aligned}
\min\{\lambda_1:\lambda_1\in\Lambda(\hV_1^T\hA\hV_1)\} &\ge \lambda_r - 2\|E\|_2\\ &> \lambda_{r+1} + 2\|E\|_2\\ &> \max\{\lambda_2:\lambda_2\in\sigma(\hV_2^T\hA\hV_2)\},
\end{aligned}
\]
which implies $\ran(\hV_1)$ is the dominant $r$-dimensional invariant subspace of $\hA$ as claimed.

Finally, it is well-known that the optimal point $\tU$ of the orthogonal Procrustes problem \eqref{eq:orthogonalProcrustesProblem} is the unitary factor in the polar decomposition of $\hV_1^TV_1.$ Since
\[
\begin{aligned}
\hV_1^TV_1 &= (I + \hX^T\hX)^{-\frac12}(V_1 + V_2\hX)^TV_1 \\
&= (I + \hX^T\hX)^{-\frac12}(V_1^TV_1 + \hX V_2^TV_1) \\
&= (I + \hX^T\hX)^{-\frac12}
\end{aligned}
\]
is symmetric and positive definite, that unitary factor is the identity. Therefore, as desired, $\hV_1$ is the closest matrix to $V_1$ in Frobenius distance among all matrices of the form $\hV_1U$, where $U\in\orth^r$. Note that this is exactly the set of matrices with orthonormal columns whose range is the dominant $r$-dimensional invariant subspace of $\hA$. 

{\it Part 2:} For the remainder of the proof $\hV_1$ is as defined in \eqref{eq:hV1} and we proceed to explicitly bound
\[
\normtwoinf{\hV_1-V_1}.
\]
We start by decomposing the error $\hV_1 - V_1$ into its components in $\ran V_1$ and $(\ran V_1)^\perp$ (recall that $\ran V_2 = (\ran V_1)^\perp$). Specifically, from~\eqref{eq:hV1} it follows that
\begin{equation}
\hV_1-V_1 = V_1((I_r + \hX^T\hX)^{-\frac12} - I_r) + V_2\hX(I+\hX^T\hX)^{-\frac12}.
\label{eq:error}
\end{equation}
We now proceed to address each part of this decomposition of the error separately.

The $(2,\infty)$-norm of the first term on the right side of \eqref{eq:error} is at most
\[
\begin{aligned}
\normtwoinf{V_1((I_r + \hX^T\hX)^{-\frac12} - I_r)} &\leq \normtwoinf{V_1} \|(I_r + \hX^T\hX)^{-\frac12} - I_r\|_2.
\end{aligned}
\]
Since 
\[
\|(I_r + \hX^T\hX)^{-\frac12} - I_r\|_2 = \lvert1-(1+\|\hX\|_2^2)^{-1/2}\rvert,
\]
we can use the fact that for any $x>0$
\begin{align*}
\lvert1-(1+x)^{-1/2}\rvert &= \left\lvert\frac{\sqrt{1+x}-1}{\sqrt{1+x}}\right\rvert \\
&\leq \left\lvert\frac{x}{\sqrt{1+x} + 1}\right\rvert\\
&\leq \frac{1}{2}x
\end{align*}
to conclude that
\[
\|(I_r + \hX^T\hX)^{-\frac12} - I_r\|_2 \leq \frac{1}{2}\|\hX\|_2^2.
\]
Furthermore, by Lemma \ref{lem:NewtonKantorovichX} $\|\hX\|_2 \le \frac{4\|E_{2,1}\|_2}{\sep_2(\Lambda_1,\Lambda_2)}$ and therefore
\[
\normtwoinf{V_1((I_r + \hX^T\hX)^{-\frac12} - I_r)} \le 8\normtwoinf{V_1}\left(\frac{\|E_{2,1}\|_2}{\sep_2(\Lambda_1,\Lambda_2)}\right)^2.
\]

At this point we have control over the first term in~\eqref{eq:error}. In addition, since the $(2,\infty)$-norm of the second term on the right side of \eqref{eq:error} is at most
\[
\normtwoinf{V_2\hX(I_r+\hX^T\hX)^{-\frac12}} \le \normtwoinf{V_2\hX}\|(I_r+\hX^T\hX)^{-\frac12}\|_2 \le \normtwoinf{V_2\hX}
\]
it follows from the triangle inequality that
\begin{equation}
\normtwoinf{\hV_1 - V_1} \le 8\normtwoinf{V_1}\left(\frac{\|E_{2,1}\|_2}{\sep_2(\Lambda_1,\Lambda_2)}\right)^2 + \normtwoinf{V_2\hX}.
\label{eq:halfMain}
\end{equation}

For the remainder of the proof we focus on bounding $\normtwoinf{V_2\hX}$. At first glance, we are tempted to appeal to the compatibility of $\normtwoinf{\cdot}$ and $\|\cdot\|_2$ to obtain
\[
\normtwoinf{V_2\hX} \le \normtwoinf{V_2}\|\hX\|_2 \le \normtwoinf{V_2}\frac{2\|E_{2,1}\|_2}{\sep_2(\Lambda_1,\Lambda_2)}.
\]
Unfortunately, this bound is generally inadequate because $\normtwoinf{V_2}$ may be much larger than $\normtwoinf{V_1}$.\footnote{In fact, we have that $\normtwoinf{\begin{bmatrix}V_1 & V_2 \end{bmatrix}}=1.$} Instead, we must study $\normtwoinf{V_2\hX}$ directly. To start, observe that $V_2\hX$ satisfies
\begin{equation}
\label{eq:hY}
0 = -V_2\hA_{2,1} + V_2\hX\hA_{1,1} - V_2\hA_{2,2}V_2^TV_2\hX + V_2\hX\hA_{1,2}V_2^TV_2\hX.
\end{equation}
In other words, $V_2\hX$ is a root of the map $G:\reals^{n\times r}\to\reals^{n\times r}$ defined as 
\[
\begin{aligned}
G:Y\to -V_2\hA_{2,1} + Y\hA_{1,1} - V_2\hA_{2,2}V_2^TY + Y\hA_{1,2}V_2^TY.
\end{aligned}
\]
Letting $\hY = V_2\hX,$ we rearrange~\eqref{eq:hY} to obtain
\[
\begin{aligned}
\hY\hA_{1,1} - V_2\Lambda_2V_2^T\hY &= -V_2\hA_{2,1} + \hY\hA_{1,2}V_2^T\hY + V_2E_{2,2}V_2^T\hY \\
&= -V_2E_{2,1} + \hY E_{1,2}V_2^T\hY + V_2E_{2,2}V_2^T\hY,
\end{aligned}
\]
where we have used that $A_{1,2}=A_{2,1}^T=0.$ 

We take norms to see that
\[
\begin{aligned}
&\normtwoinf{\hY\hA_{1,1} - V_2\Lambda_2V_2^T\hY} \\
&\quad\le \normtwoinf{V_2E_{2,1}} + \normtwoinf{\hY E_{1,2}V_2^T\hY} + \normtwoinf{V_2E_{2,2}V_2^T\hY} \\
&\quad\le \normtwoinf{V_2E_{2,1}} + \normtwoinf{\hY}\|\hY^TV_2E_{1,2}^T\|_2 + \normtwoinf{V_2E_{2,2}V_2^T\hY} \\
&\quad\le \normtwoinf{V_2E_{2,1}} + \normtwoinf{\hY}\|V_2E_{1,2}^T\|_2\|\hY\|_2 + \normtwoinf{V_2E_{2,2}V_2^T\hY}.
\end{aligned}
\]
It now follows from Lemma~\ref{lem:NewtonKantorovichX} that
\[
\begin{aligned}
&\normtwoinf{\hY\hA_{1,1} - V_2\Lambda_2V_2^T\hY} \\
&\quad\le \normtwoinf{V_2E_{2,1}} + \normtwoinf{\hY}\frac{2\|E_{2,1}\|_2^2}{\gap} + \normtwoinf{V_2E_{2,2}V_2^T\hY}.
\end{aligned}
\]
Next, observe that the left side is at least
\[
\begin{aligned}
\normtwoinf{\hY\hA_{1,1} - V_2\Lambda_2V_2^T\hY} &\ge \normtwoinf{\hY\Lambda_1 - V_2\Lambda_2V_2^T\hY} - \normtwoinf{\hY}\|E_{1,1}\|_2 \\
&\ge \sep_{(2,\infty),V_2}(\Lambda_1,V_2\Lambda_2V_2^T)\normtwoinf{\hY} - \normtwoinf{\hY}\|E_{1,1}\|_2 \\
&\ge \sep_{(2,\infty),V_2}(\Lambda_1,V_2\Lambda_2V_2^T)\normtwoinf{\hY} - \normtwoinf{\hY}\|E\|_2 \\
&\ge \frac34\gap\normtwoinf{\hY},
\end{aligned}
\]
and, therefore,
\[
\left(\frac{3}{4}\gap - \frac{2\|E_{2,1}\|_2^2}{\gap}\right)\normtwoinf{\hY}  \le \normtwoinf{V_2E_{2,1}} + \normtwoinf{V_2V_2^TE\hY}.
\]
Since $2\|E_{2,1}\|_2/\gap \le 1$ and $\|E_{2,1}\|_2\le \gap/4$ by assumption (using $\|E_{2,1}\|_2\leq\|E\|_2$) we have that
\begin{equation}
\label{eqn:thm_split}
\normtwoinf{\hY} \le \frac{2\normtwoinf{V_2E_{2,1}}}{\gap} + \frac{2\normtwoinf{V_2V_2^TE\hY}}{\gap}.
\end{equation}

Prior to concluding the proof, we summarize our results up to this point in Lemma~\ref{lem:Y}. We partly pause to highlight a natural launching point for problem specific analysis, particularly in settings where it is possible to control $\normtwoinf{V_2V_2^TE\hY}$ in a tighter manner than suggested by our worst-case bounds that follow.    
\begin{lemma}
\label{lem:Y}
Let $A\in\reals^{n\times n}$ be symmetric with an eigen-decomposition $A = V_1\Lambda_1V_1^T +V_2\Lambda_2V_2^T$ following the conventions of~\eqref{eqn:eigen-decomp} and
\[
\gap = \min\{\sep_2(\Lambda_1,\Lambda_2),\sep_{(2,\infty),V_2}(\Lambda_1,V_2\Lambda_2V_2^T)\}.
\]
If $\|E\|_2 \le \frac{\gap}{5}$ then
\begin{equation*}
\normtwoinf{\hV_1\tU - V_1} \le 8\normtwoinf{V_1}\left(\frac{\|E_{2,1}\|_2}{\sep_2(\Lambda_1,\Lambda_2)}\right)^2 + 2\frac{\normtwoinf{V_2E_{2,1}}+\normtwoinf{V_2E_{2,2}V_2^T\hY}}{\gap},
\end{equation*}
where $\hV_1$ is any matrix with orthonormal columns whose range is the dominant $r$-dimensional invariant subspace of $\hA,$ $\tU$ solves the orthogonal Procrustes problem
\begin{equation*}
\min\{\|\hV_1U - V_1\|_F:U\in\orth^r\},
\end{equation*}
and $\hY = V_2\hX$ where $\hX$ is the root of
\[
F:X\to -\hA_{2,1} + X\hA_{1,1} - \hA_{2,2}X + X\hA_{1,2}X
\]
found by the iteration~\eqref{eq:NewtonIterationX} starting at $X_0=0$ and thereby satisfying Lemma~\ref{lem:NewtonKantorovichX}.

\end{lemma}

Moving forward, Lemma~\ref{lem:Y} immediately implies that
\[
\normtwoinf{\hY} \le 2\frac{\normtwoinf{V_2E_{2,1}}}{\gap} + 4\frac{\normtwoinf{V_2E_{2,2}V_2^T}\|E_{2,1}\|_2}{\gap\times \sep_2(\Lambda_1,\Lambda_2)},
\]
where we have used the sub-multiplicative relationships
\[
\normtwoinf{V_2E_{2,2}V_2^T\hY} \leq \normtwoinf{V_2E_{2,2}V_2^T}\|\hY\|_2
\] 
in conjunction with Lemma~\ref{lem:NewtonKantorovichX} to bound $\|\hY\|_2$. This concludes the proof of~\eqref{eq:mainP} in Theorem~\ref{thm:main} and Corollary~\ref{cor:main} follows immediately. 

We now briefly retrace our steps to prove Corollary~\ref{cor:infBound}. In particular, returning to Lemma~\ref{lem:Y} we can instead conclude that
\[
\normtwoinf{\hY}\left(1 -\frac{2(1+\mu^2)}{\gap}\frac{\normtwoinf{E\hY}}{\normtwoinf{\hY}}\right) \le \frac{2\normtwoinf{V_2E_{2,1}}}{\gap},
\]
where we have used the observation that 
\[
\begin{aligned}
\|V_2V_2^T\|_\infty &= \|I_n - V_1V_1^T\|_\infty \\
&\le 1 + \max\{{\textstyle\sum_{j=1}^n|v_i^Tv_j|}:i\in[n]\} \\
&\le 1 + n\normtwoinf{V_1}^2\\
&\le 1 + \mu^2,
\end{aligned}
\]
and that $\hY\in\ran{V_2}.$
Now, since $\normtwoinf{E\hY}/\normtwoinf{\hY} \leq \|E\|_{\infty}$ if we further assume that $\|E\|_{\infty}\leq \gap / (4+ 4\mu^2)$ we get that
\[
\left(1 -\frac{2(1+\mu^2)}{\gap}\frac{\normtwoinf{E\hY}}{\normtwoinf{\hY}}\right) \ge \frac{1}{2}.
\]
Therefore, 
\[
\normtwoinf{\hY} \le 4\frac{\normtwoinf{V_2E_{2,1}}}{\gap},
\]
which concludes the proof of Corollary~\ref{cor:infBound}.

\subsection{Observations and implications}
\label{sec:observations}
We now discuss several aspects of our bounds in greater detail. In particular, we first construct specific examples that show any of the 3 terms in the bound of Theorem~\ref{thm:main} may tightly control the error and therefore are all necessary. We then argue why $\sep_{(2,\infty),V_2}(\Lambda_1,V_2\Lambda_2V_2^T)$ should be directly included in our bounds by showing that in the worst case it may be considerably smaller than $\sep_F(\Lambda_1,\Lambda_2).$ Lastly, we discuss the use of our bound in certain probabilistic scenarios motivated by applications and highlight how our bounds can facilitate further analysis of those situations.

\subsubsection{When the upper bound is tight}
The first term of our upper bound represents the projection of the error onto $V_1$ while the latter two terms arise from the projection onto $V_2.$ Therefore, we focus on the latter piece to understand if both the terms are necessary and examine our potentially loose use of the triangle inequality and sub-multiplicative bounds in the proof. To accomplish this, we construct a specific example and examine the behavior of our bound. 

We bounded the projection of the error onto $\ran V_2$ as 
\begin{equation}
\label{eqn:tight}
\normtwoinf{V_2V_2^T(\hV_1-V_1)}\leq\frac{2\normtwoinf{V_2E_{2,1}}}{\gap} + 4\frac{\normtwoinf{V_2E_{2,2}V_2^T}\|E_{2,1}\|_2}{\gap\times \sep_2(\Lambda_1,\Lambda_2)}.
\end{equation}
While the first term on the right hand side of~\eqref{eqn:tight} is a natural part of our bound given the quadratic form~\eqref{eq:hY}, the second term arose from the sub-multiplicative bound
\[
\normtwoinf{V_2E_{2,2}V_2^T\hY} \leq \normtwoinf{V_2E_{2,2}V_2^T}\|\hY\|_2 \leq \frac{2}{\sep_2(\Lambda_1,\Lambda_2)}\normtwoinf{V_2E_{2,2}V_2^T}\|E_{2,1}\|_2.
\] 
Nevertheless, both terms are necessary\textemdash there are perturbations that saturate each part of the bound.

To show this, we build an example that demonstrates two clear regimes\textemdash one where the first term of~\eqref{eqn:tight} controls the error tightly and one where the second term does. We accomplish this by picking $E$ such that for the resulting $\hY$ 
\[
\normtwoinf{E\hY} \approx \normtwoinf{V_2E_{2,2}V_2^T}\frac{\|E\|_2}{\gap}.
\]
To keep thing simple, we consider the $r=1$ case with $\lambda_1 = 1$ and $\lambda_2,\ldots,\lambda_n = 0,$ which implies $\gap = 1.$ We then let $V_1 = \mathbf{1}/\sqrt{n}$ and observe that if $E_{1,2} = 0$ and $E_{1,1} = 0$ then $\hY$ satisfies
\[
(I-V_2E_{2,2}V_2^T)\hY = V_2E_{2,1}.
\]
The core insight in our construction is that we can now choose $V_2E_{2,2}V_2^T$ and $V_2E_{2,1}$ carefully to accomplish our goal. This is because we can essentially determine $\hY$ (in fact, to first order it looks like $V_2E_{2,1}$ if the norm of $E$ is sufficiently small).

Now, let 
\[
V_2E_{2,2}V_2^T = V_2V_2^T(e_1\mathbf{1}_{\pm}^T/\sqrt{n} + \mathbf{1}_{\pm}e_1^T/\sqrt{n})V_2V_2^T.
\]
In this case there exists a $y$ with $y_1=1$ and $y_2,\ldots,y_n = \mathcal{O}(1/\sqrt{n})$ such that $(I-V_2E_{2,2}V_2^T)y = \mathcal{O}(1/\sqrt{n})$ entry-wise. Setting $V_2E_{2,1}$ to be proportional to $V_2V_2^T(I-V_2E_{2,2}V_2^T)y$ lets us deterministically construct an $E$ where $\hY$ essentially saturates the sub-multiplicative bound.\footnote{Practically one can make $E$ symmetric by setting $E_{1,2}$ appropriately without destroying the example and scale $E$ by $n^{-1/3}$ so that we expect convergence in $n$. Details are available in the online materials referred to in the numerical experiments section. Choices of scaling constants in individual parts of $E$ control where the crossover point occurs between the two bounds.} The preceding construction yields a purely deterministic example illustrating in Figure~\ref{fig:twoboundsV2} that either part of~\eqref{eqn:tight} can be dominant. Similarly, Figure~\ref{fig:twoboundsV1} shows, as expected, that our bound on the projection of the error onto $V_1$ tightly captures the asymptotic behavior.

\begin{figure}[ht]
    \begin{subfigure}[t]{0.49\columnwidth}
	  \includegraphics[width=1\columnwidth]{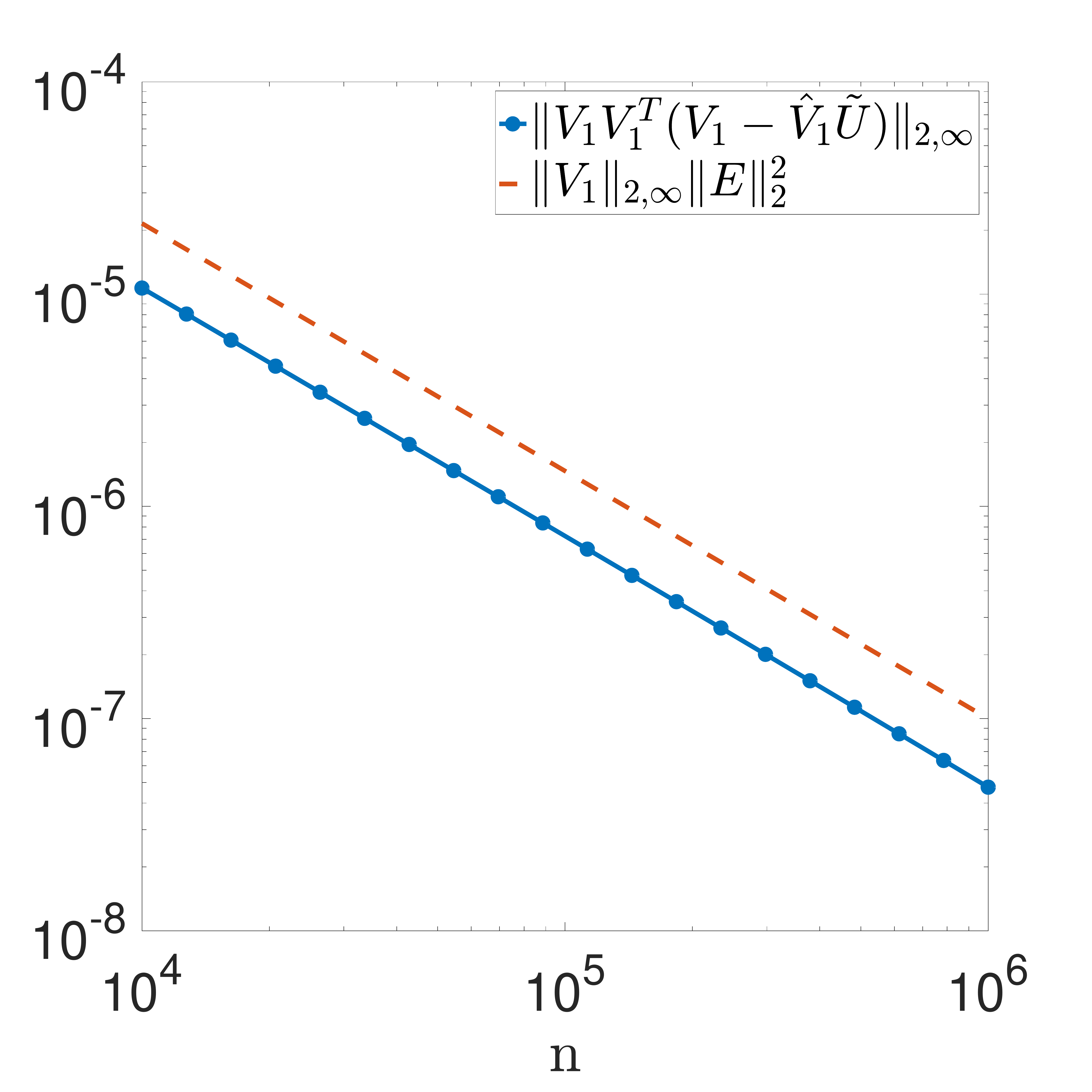}
	  \caption{Projection onto $V_1$}
	  \label{fig:twoboundsV1}
  \end{subfigure}
  \begin{subfigure}[t]{0.49\columnwidth}
	  \includegraphics[width=1\columnwidth]{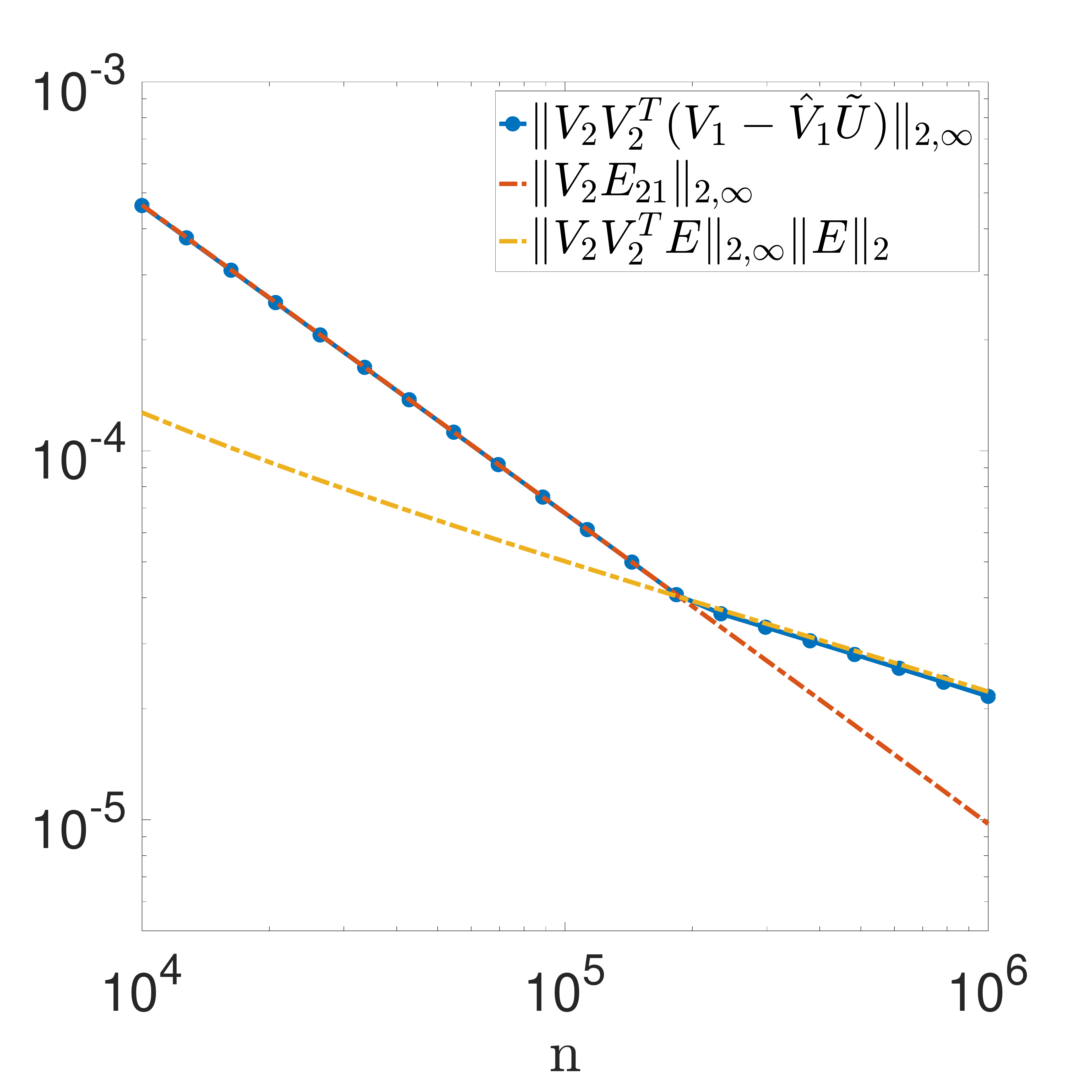}
	  \caption{Projection onto $V_2$}
	  \label{fig:twoboundsV2}
  \end{subfigure}
  \caption{Asymptotic behavior of $\min_{\tU=\pm 1}\normtwoinf{\hV_1 - V_1\tU}$ split into the component of the error in $\ran V_1$ and $\ran V_2.$ This example shows that either part of the upper bound in Theorem~\ref{thm:main} associated with the error projected onto $\ran V_2$ can tightly control the rate of decay. Similarly, our control over the projection of the error onto $\ran V_1$ matches the observed rate for this example. Note that, as indicated by the legend, we have not included constants (they would appear to be slightly loose in this case) and used the upper bound $\normtwoinf{V_2E_{2,2}V_2^T}\leq\normtwoinf{V_2V_2^TE}.$ Therefore the dotted lines technically represent our upper bounds to within small constants.}
  \label{fig:twobounds}
\end{figure}

\subsubsection{Inclusion of norm specific separation}
In the proof of Theorem~\ref{thm:main} $\sep_{(2,\infty),V_2}(\Lambda_1,V_2\Lambda_2V_2^T)$ arises somewhat naturally. Nevertheless, ideally one would be able to generically relate it tightly to traditional notions of an eigengap. Unfortunately, the lower bound provided in Lemma~\ref{lem:sepU_lowerbound} is essentially tight. To show this we explicitly construct an example that achieves (to within a small constant) the lower bound $\sep_{(2,\infty),V_2}(\Lambda_1,V_2\Lambda_2V_2^T) \geq \sep_F(\Lambda_1,\Lambda_2)/\sqrt{n}$.

Assume $n$ is even and let $\mathbf{1}$ be the vector of all ones and $(\mathbf{1}_{\pm})_i = -1$ if $i > n/2$ and $1$ otherwise. Now, define $v_1 = \begin{bmatrix} 0 & \mathbf{1}^T\end{bmatrix}^T/\sqrt{n}$ and $v_2 = \begin{bmatrix} 0 & \mathbf{1}_{\pm}^T\end{bmatrix}^T/\sqrt{n}$ and consider the $(n+1)\times (n+1)$ matrix  
\[
A = 2*v_1v_1^T + \begin{bmatrix}e_1 & v_2\end{bmatrix}\begin{bmatrix}0 & 1\\ 1 & 0\end{bmatrix}\begin{bmatrix}e_1 & v_2\end{bmatrix}^T
\]
In our framework this corresponds to setting $\Lambda_1 = 2,$ $\Lambda_2 = \diag(1,-1,0,\ldots,0)\in\reals^{n\times n},$ and letting $V_2$ to be any matrix with orthonormal columns spanning the orthogonal complement of $v_1$ such that 
\[
\begin{bmatrix}e_1 & v_2\end{bmatrix}\begin{bmatrix}0 & 1\\ 1 & 0\end{bmatrix}\begin{bmatrix}e_1 & v_2\end{bmatrix}^T = V_2\Lambda_2V_2^T.
\]

In this case, by picking the vector $q = e_1 +2v_2$ (which is in the range of $V_2$ and satisfies $\normtwoinf{q} = 1$ if $n\geq 4$) we see that 
\begin{align*}
\normtwoinf{q\Lambda_1 - V_2\Lambda_2 V_2^Tq} &= \normtwoinf{2q - 2e_1 - v_2}\\
&= \normtwoinf{3v_2} \\
& = \frac{3}{\sqrt{n}}.
\end{align*}
Therefore, $\sep_{(2,\infty),V_2}(\Lambda_1,V_2\Lambda_2V_2^T) \leq 3/\sqrt{n}$ and since $\sep_F(\Lambda_1,\Lambda_2)=1$ this shows that Lemma~\ref{lem:sepU_lowerbound} is essentially tight. Nevertheless, Lemmas~\ref{lem:sepU_lowerbound} and~\ref{lem:sepU_normlowerbound} also show situations where $\sep_{(2,\infty),V_2}$ and $\sep_F$ are more closely related. Ultimately, the range of possible relationships between $\sep_{(2,\infty),V_2}$ and $\sep_F$ motivates the inclusion of $\sep_{(2,\infty),V_2}$ directly in any worst-case deterministic bound.

\subsubsection{Probabilistic settings}
\label{subsec:probSetting}
While the two preceding sections illuminate why various terms in our constructed bounds are necessary, one may expect that in random settings these terms can be controlled more effectively and the expected behavior may be far from the worst case. In particular, if we consider $E = \sigma Z$ where $Z$ is symmetric and $Z_{i,j}$ are \iid\ $ \mathcal{N}(0,1)$ (up to the symmetry constraint) random variables we have by standard properties of Gaussian random matrices (see, \emph{e.g.,} \cite[\S~4.4]{vershynin2018HighDimensional})
\[
\Pr(\|E\|_2 > 3\sigma\sqrt{n}) \lesssim e^{-\frac{n}{2}}.
\]
Furthermore, under an assumption that $V_1$ is incoherent (recall that $\mu = \sqrt{n}\normtwoinf{V_1}$) and $E$ is independent from $V_1$ we may assert (again via standard properties of Gaussian random matrices and a union bound~\cite[\S~4.4]{vershynin2018HighDimensional}) that 
\[
\normtwoinf{E_{2,1}} \le (1+\mu^2)\normtwoinf{EV_1} \lesssim \sigma\sqrt{\log n}
\]
with high probability. Ideally, these results would directly imply that in such a setting 
\begin{equation}
\label{eqn:conjecture}
\min\{\normtwoinf{\hV_1U - V_1}:U\in\orth^r\} \lesssim \sigma\sqrt{\log n}
\end{equation}
with high probability. Unfortunately, this does not follow directly from Theorem~\ref{thm:main} as for certain values of $\sigma$ the error bound is dominated by the term $\normtwoinf{E}\|E\|_2 \lesssim \sigma^2 n.$ 

Figure \ref{fig:twobounds} shows that this is not an artifact of our analysis; it is possible to construct examples that saturate the error bound. However, these examples are adversarial. In particular, the independence among the entries of $Z$ permits more direct control of $\normtwoinf{E\hY}$ and \cite{abbe2017Entrywise} appeal to a leave-one-out technique to achieve such direct control. Based on our analysis, we conjecture that their column-wise independence condition may be relaxed to independence among the $E_{i,j}$'s ($i,j\in\{1,2\}$). While we defer a thorough analysis of this problem in the probabilistic setting for future work, Theorem~\ref{thm:probG} illustrates how additional mild assumptions on $E$ can be used to improve our bounds. While we have stated Theorem~\ref{thm:probG} for the Procrustes solution, the result for the minimum over $U\in\orth^r$ (analogous to Corollary~\ref{cor:main}) follows immediately.
\begin{theorem}
\label{thm:probG}
Let $A\in\reals^{n\times n}$ be symmetric with eigen-decomposition 
\[
A = V_1\Lambda_1V_1^T +V_2\Lambda_2V_2^T
\] 
following the conventions of~\eqref{eqn:eigen-decomp}, 
\[
\gap = \min\{\sep_2(\Lambda_1,\Lambda_2),\sep_{(2,\infty),V_2}(\Lambda_1,V_2\Lambda_2V_2^T)\}
\] 
be independent of $n,$ and $\normtwoinf{V_1}\lesssim 1/\sqrt{n}.$ If $E = \sigma Z$ where $Z$ is symmetric, $Z_{i,j}$ are \iid\ $\mathcal{N}(0,1)$ random variables (up to the symmetry constraint) and $\sigma \lesssim 1/\sqrt{n}$ then with probability $1-o(1)$
\begin{align*}
\min\{\normtwoinf{\hV_1\tU - V_1}:U\in\orth^r\} &\lesssim \sigma^2 \sqrt{n} + \sigma \sqrt{\log n} +(\sigma \sqrt{n})^3,
\end{align*}
where $\hV_1$ is any matrix with orthonormal columns whose range is the dominant $r$-dimensional invariant subspace of $\hA,$ and $\tU$ solves the orthogonal Procrustes problem
\begin{equation*}
\min\{\|\hV_1U - V_1\|_F:U\in\orth^r\}.
\end{equation*}
\end{theorem}
\begin{proof}
The key idea behind this proof is to consider a second perturbation $\tE$ drawn from the same distribution as $E$ conditioned on $\tE_{1,2} = E_{1,2}$ and $\tE_{1,1}=E_{1,1}$\emthin nevertheless, $E_{2,2}$ and $\tE_{2,2}$ are still independent. Given $E$ and $\tE$ appropriate control of $\|E\|_2$ and $\|\tE\|_2$~\cite[\S~4.4]{vershynin2018HighDimensional} with probability $1-o(1)$ allow us to invoke Lemma~\ref{lem:prob} to get that
\begin{align*}
\normtwoinf{\hV_1\tU - V_1} &\lesssim \frac{1}{\sqrt{n}}\|E_{2,1}\|_2^2 + \normtwoinf{V_2E_{2,1}}+\normtwoinf{V_2E_{2,2}V_2^T\tY}\\ &\phantom{\lesssim} + \|E_{2,2}\|_2\normtwoinf{V_2E_{2,2}V_2^T}\|E_{2,1}\|_2+\normtwoinf{V_2E_{2,2}V_2^T}\|E_{2,1}\|_2^3
\end{align*}
where $\tY\in\reals^{n\times r}$ is in the range of $V_2$ and satisfies
\[
0 = -V_2\tA_{2,1} + \tY\tA_{1,1} - V_2\tA_{2,2}V_2^T\tY + \tY\tA_{1,2}V_2^T\tY.
\]
This implies that $\tY$ is independent of $E_{2,2}$ and satisfies $\|\tY\|_2\lesssim \|E_{2,1}\|_2.$ The stated bound follows immediately from the following bounds on $E$ that all hold with probability $1-o(1)$~\cite[\S~4.4]{vershynin2018HighDimensional}
\begin{align*}
\|E_{2,1}\|_2 &\lesssim \sigma\sqrt{n}\\
\normtwoinf{V_2E_{2,1}} &\lesssim \sigma\sqrt{\log n}\\
\normtwoinf{V_2E_{2,2}V_2^T\tY} &\lesssim \sigma^2\sqrt{n \log n}\\
\|E_{2,2}\|_2 &\lesssim \sigma\sqrt{n}\\
\normtwoinf{V_2E_{2,2}V_2^T} &\lesssim \sigma\sqrt{n},\\
\end{align*}
where because $\normtwoinf{V_1}\lesssim 1/\sqrt{n}$ we have that $\|V_2V_2^T\|_{\infty} \lesssim 1.$
\end{proof}
\begin{remark}
If we consider $\mu$ to be constant in $n$ (synonymous with $\normtwoinf{V_1}\lesssim 1/\sqrt{n}$), the rate of convergence of $\min\{\normtwoinf{\hV_1U - V_1}:U\in\orth^r\}$ implied by Theorem~\ref{thm:probG} is faster than that given by direct application Theorem~\ref{thm:main}. In particular, we get that  $\min\{\normtwoinf{\hV_1U - V_1}:U\in\orth^r\}\rightarrow 0$ as $(\sigma \sqrt{n})^3$ vs $(\sigma \sqrt{n})^2.$ While an improvement, we believe more intricate probabilistic techniques will be necessary to prove an upper bound that matches our conjectured rate of $\sigma \sqrt{\log n}$ (see~\eqref{eqn:conjecture}) in this setting.
\end{remark}

\section{Numerical simulations}
\label{sec:numerics}
We now provide numerical simulations to illustrate the effectiveness of our bounds and elaborate on a key difference between them and prior work. We consider two settings, one where $A$ is low-rank and one where $A$ is not low-rank and $\normtwoinf{V_2\Lambda_2V_2^T}$ is constant with respect to $n.$  In all these experiments, and as before, we let $\mathbf{1}$ be the vector of all ones and $(\mathbf{1}_{\pm})_i = -1$ be $1$ in the first half of the entires and $-1$ in the second half. We also let $E=\sigma Z$ where $Z$ is a symmetric matrix whose entries are \iid\ $\mathcal{N}(0,1).$ Code to generate these plots (and Figure~\ref{fig:twobounds}) is available at \texttt{https://github.com/asdamle/rowwise-perturbation}.

\subsection{A is low-rank}
Assume $n$ is even, let 
\[
V_1 = \frac{1}{\sqrt{n}}\begin{bmatrix}\mathbf{1} & \mathbf{1}_{\pm}\end{bmatrix},
\]
and consider $A = V_1V_1^T.$ In this setting, $\gap =1$ and if $\sigma = 1/n$ Theorem~\ref{thm:main} shows that $\min\{\normtwoinf{\hV_1U - V_1}:U\in\orth^r\} \lesssim_P \frac{\sqrt{\log{n}}}{n},$ where the term controlling the rate with respect to $n$ is $\normtwoinf{V_2E_{2,1}}.$ Furthermore, Theorem~\ref{thm:main} shows that $\normtwoinf{\hV_1\tU - V_1} \lesssim_P \frac{\sqrt{\log{n}}}{n}.$ Therefore, we conduct an experiment where for increasing values of $n$ we construct $\hA$ for several instances of $E,$ compute $\hV_1$ and the solution to the orthogonal Procrustes problem, and measure $\normtwoinf{\hV_1\tU - V_1}.$ Figure~\ref{fig:scale_n12} clearly shows the expected behavior for both the traditional subspace distance and our bound on $\normtwoinf{\cdot}.$ We also explicitly include our upper bound from Theorem~\ref{thm:main}, showing that it tightly describes the behavior to within constants. Lastly, Figure~\ref{fig:scale_n12} shows a clear distinction between our upper bound and simply using Davis-Kahan to upper bound $\normtwoinf{\hV_1\tU - V_1}$ (viable since $\normtwoinf{\cdot}\leq\|\cdot\|_2$)\emthin it is not just a matter of constants, in this case the rate of decay with respect to $n$ is fundamentally different for the two bounds.

Perhaps more interestingly, we also consider the case where $E = (1/n^{3/4})Z.$ In this case, simply applying concentration bounds to Theorem~\ref{thm:main} predicts $\normtwoinf{\hV_1\tU - V_1} \lesssim_P \frac{\sqrt{\log n}}{\sqrt{n}}.$ However, as expected in this setting, the bound used to control $\normtwoinf{E\hY}$ is loose and Figure~\ref{fig:scale_n14} shows that the error acts as if $E$ and $\hY$ were independent (though they are decidedly not) yielding an observed convergence rate of $\frac{\sqrt{\log n}}{n^{3/4}}.$ Nevertheless, as before we explicitly compute and plot the upper bound from Theorem~\ref{thm:main} for comparison. Notably, in this setting the improved bound of Theorem~\ref{thm:probG} is sufficient to correctly capture the observed asymptotic behavior.

\begin{figure}[ht]
  \centering 
  \begin{subfigure}[t]{0.49\columnwidth}
	  \includegraphics[width=1\columnwidth]{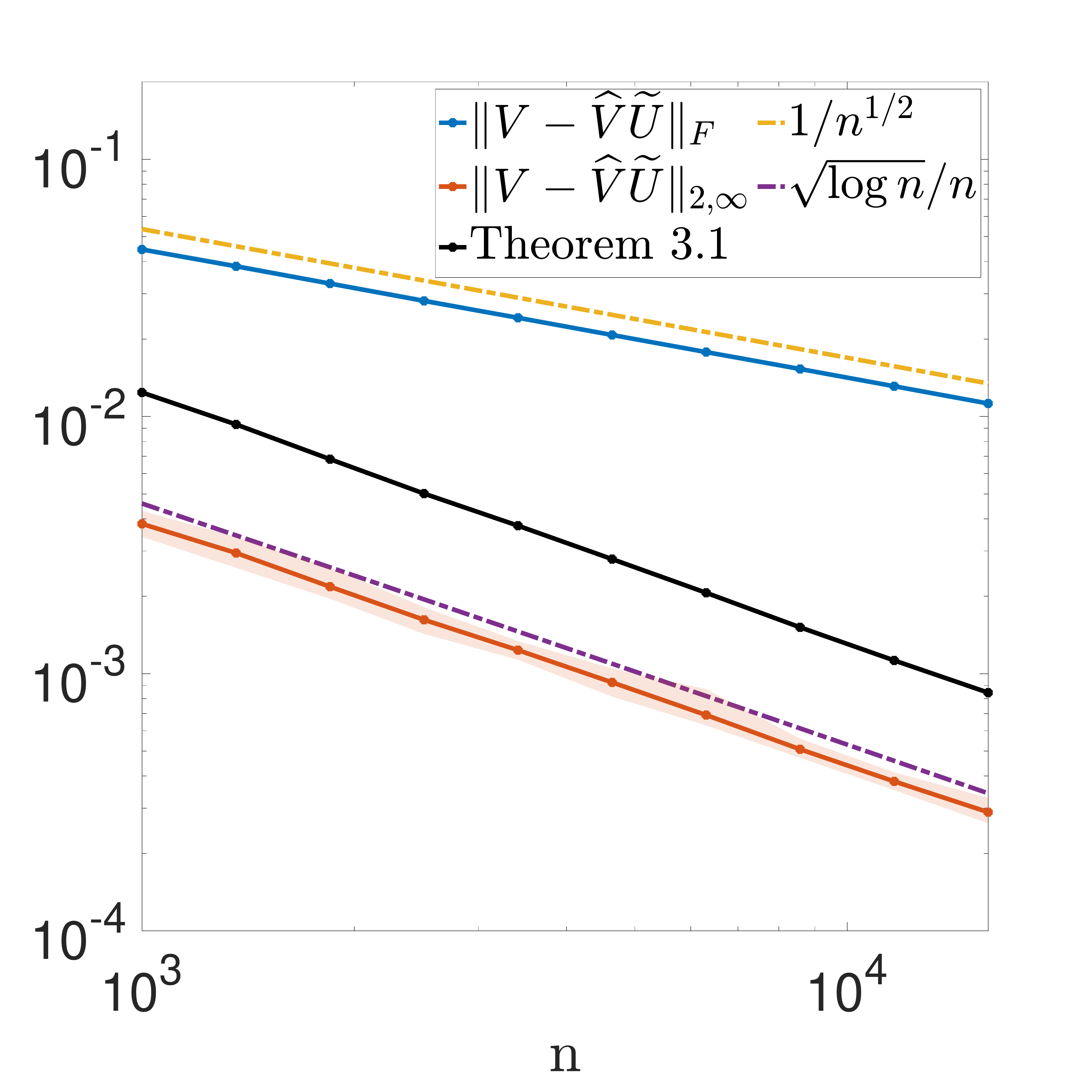}
	  \caption{$E = (1/n)Z$}
	  \label{fig:scale_n12}
  \end{subfigure}
  \begin{subfigure}[t]{0.49\columnwidth}
	  \includegraphics[width=1\columnwidth]{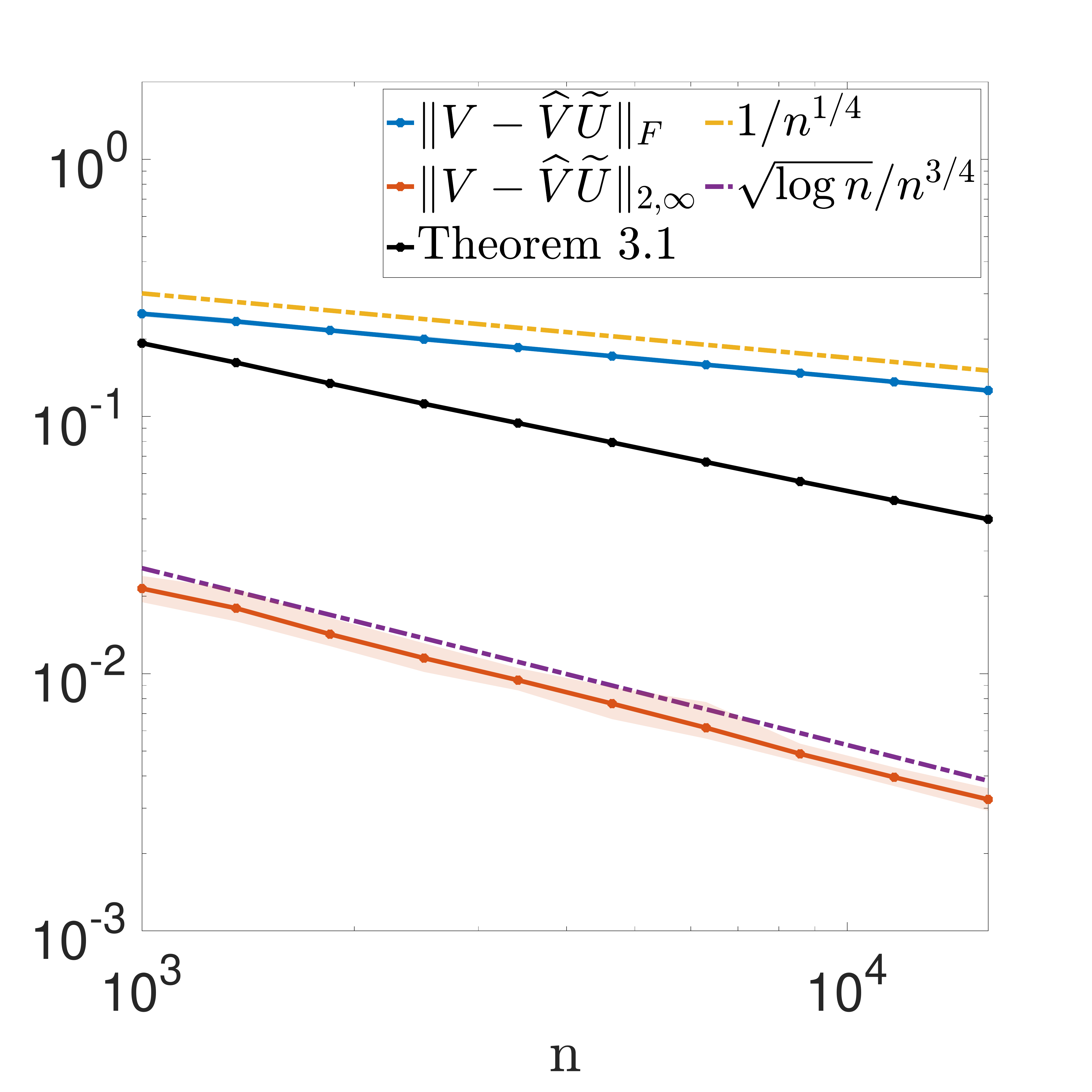}
	  \caption{$E = (1/n^{3/4})Z$}
	  \label{fig:scale_n14}
  \end{subfigure}
  \caption{Asymptotic behavior of $\normtwoinf{\hV_1\tU - V_1}$ where $A$ is low-rank and $E = \sigma Z$ for varying $\sigma$ (in this case $\mathbb{E}[\|E\|_2]\sim \sigma \sqrt{n}$), along with the upper bound given by Theorem~\ref{thm:main} for each instance of the problem. For each $n$ the experiment was repeated 30 times and we report the mean of the error and computed upper bounds. The shaded regions represent the area between the 0.05 and 0.95 quantiles of the respective lines. We see that the bounds provided by theorems~\ref{thm:main} and~\ref{thm:probG} are sharp (modulo a multiplicative constant). In the plot on the left, the slope of red line ($\normtwoinf{\hV_1\tU - V_1}$) matches the slope of the black line (bound provided by Theorem~\ref{thm:main}). In the plot on the right, the slope of the red line matches the slope of the purple line (rate provided by Theorem~\ref{thm:probG}). For reference we provide the more slowly converging quantity $\|\hV_1\tU - V_1\|_F,$ which is within constants of $\dist{(\hV_1,V_1)}$ and behaves as predicted by classical theory such as the Davis-Kahan Theorem~\cite{davis1970rotation}.}
\end{figure}

\subsection{A is not low-rank}
Next, we consider the case where $A$ itself is no longer low-rank. For even $n$ let 
\[
V_1 = \frac{1}{\sqrt{n}}\begin{bmatrix}\mathbf{1} & \mathbf{1}_{\pm}\end{bmatrix},
\]
$v_2 = e_1-e_2 \in\reals^{n\times 1},$ and
\[
A = 4V_1V_1^T + v_2v_2^T.
\]
Notably, $A$ is no longer low-rank and the component of $A$ orthogonal to $V_1$ is coherent, of significant relative magnitude, and does not decay with $n.$ Nevertheless, our results immediately imply that the asymptotic behavior of $\normtwoinf{\hV_1\tU - V_1}$ should match that of the low-rank case.\footnote{Here $\sep_{(2,\infty)}(\Lambda_1,V_2\Lambda_2V_2^T) \geq 2$ as a consequence of Lemma~\ref{lem:sepU_normlowerbound}} In contrast, this behavior is not accurately predicted by the upper bounds from~\cite{cape2017Twotoinfinity} reproduced in~\eqref{eqn:cape2017}. Using the same experimental set up as before, Figures~\ref{fig:coherent_n12} and~\ref{fig:coherent_n14} clearly illustrate the asymptotic behavior we expect\textemdash mirroring that of Figures~\ref{fig:scale_n12} and~\ref{fig:scale_n14} respectively. We also compute and plot our upper bound from Theorem~\ref{thm:main} for reference, observing that it is once again descriptive. Inclusion of the Davis-Kahan bound shows, once again, that there is a clear distinction between $\normtwoinf{\hV_1\tU - V_1}$ and $\|\hV_1\tU - V_1\|_F.$

\begin{figure}[ht]
  \centering 
  \begin{subfigure}[t]{0.49\columnwidth}
	  \includegraphics[width=1\columnwidth]{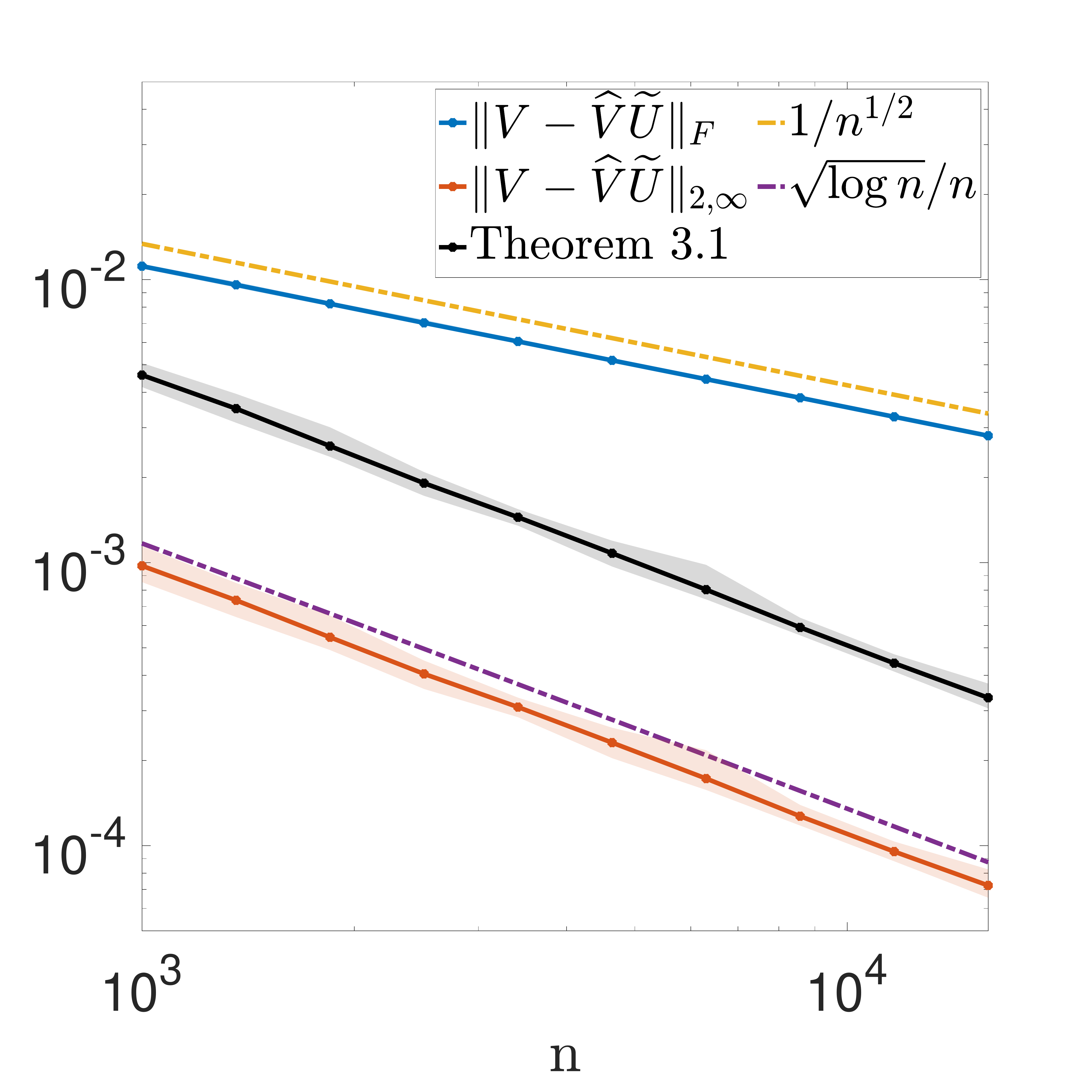}
	  \caption{$E = (1/n)Z$}
	  \label{fig:coherent_n12}
  \end{subfigure}
  \begin{subfigure}[t]{0.49\columnwidth}
	  \includegraphics[width=1\columnwidth]{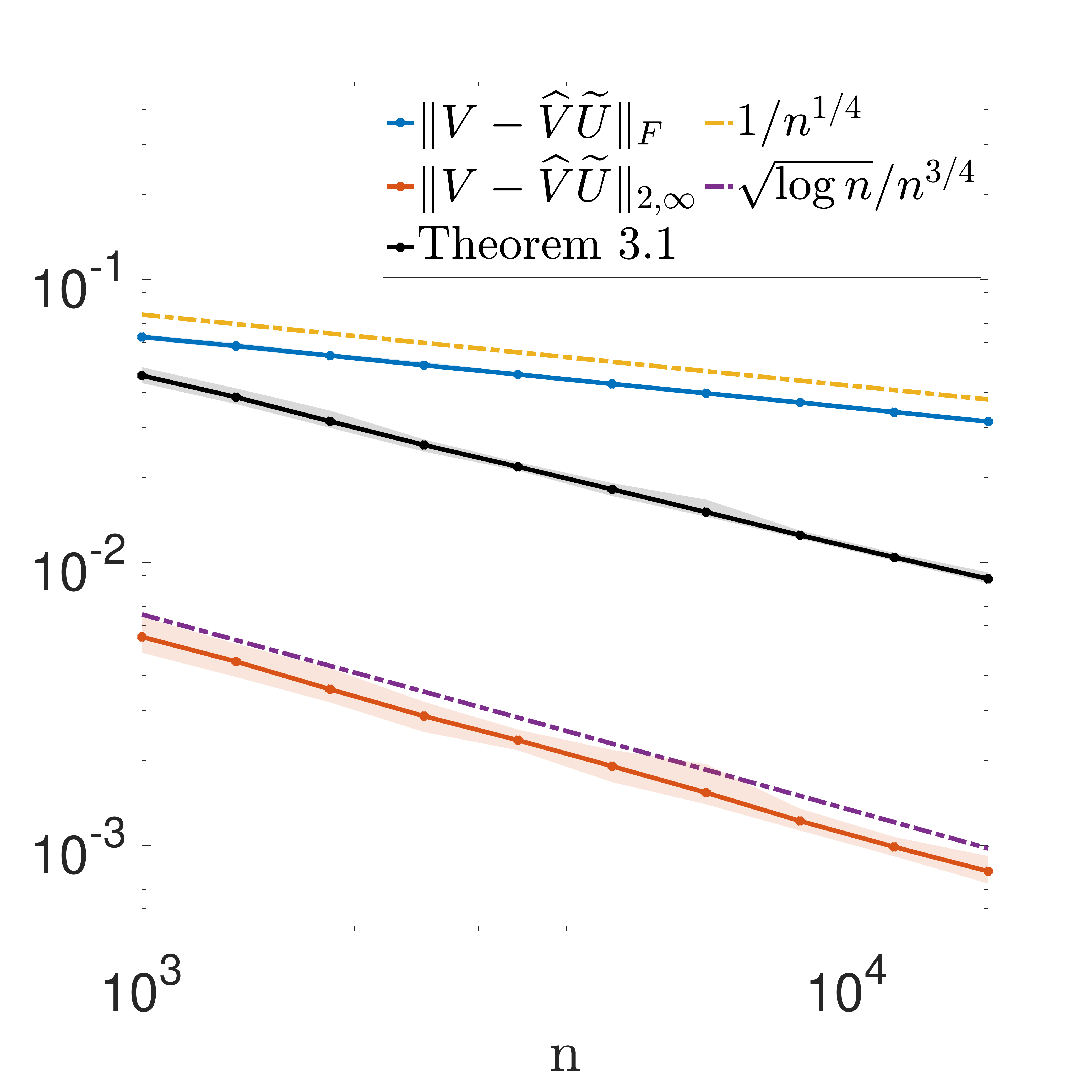}
	  \caption{$E = (1/n^{3/4})Z$}
	  \label{fig:coherent_n14}
  \end{subfigure}
  \caption{Asymptotic behavior of $\normtwoinf{\hV_1\tU - V_1}$ where $A$ is not low-rank and $E = \sigma Z$ for varying $\sigma$ (in this case $\mathbb{E}[\|E\|_2]\sim \sigma \sqrt{n}$), along with the upper bound given by Theorem~\ref{thm:main} for each instance of the problem. For each $n$ the experiment was repeated 30 times and we report the mean of the error and computed upper bounds. The shaded regions represent the area between the 0.05 and 0.95 quantiles of the respective lines. We see that the bounds provided by theorems~\ref{thm:main} and~\ref{thm:probG} are sharp (modulo a multiplicative constant). In the plot on the left, the slope of red line ($\normtwoinf{\hV_1\tU - V_1}$) matches the slope of the black line (bound provided by Theorem~\ref{thm:main}). In the plot on the right, the slope of the red line matches the slope of the purple line (rate provided by Theorem~\ref{thm:probG}). For reference we provide the more slowly converging quantity $\|\hV_1\tU - V_1\|_F,$ which is within constants of $\dist{(\hV_1,V_1)}$ and behaves as predicted by classical theory such as the Davis-Kahan Theorem~\cite{davis1970rotation}.}
\end{figure}

\section{Extensions of our bounds for non-normal matrices}
\label{sec:nonnormal}
We have constructed our bounds for symmetric matrices $A$ subject to arbitrary additive symmetric perturbations $E$. Nevertheless, they may be extended in several directions for non-normal $A$ and/or $E$ by considering more general invariant subspaces arising in the Schur forms of $A$ and $A+E.$ We briefly articulate how such extensions are readily obtained following the same proof strategy used for Theorem~\ref{thm:main}.

\subsection{Schur form subspaces}
Our results can be directly extended to the Schur form for non-normal matrices; notably, we also no longer require any assumptions on $E$. We now let
\begin{equation}
\label{eqn:SchurForm}
A = \begin{bmatrix}U_1 & U_2\end{bmatrix}\begin{bmatrix}T_{1,1} & T_{1,2} \\ 0 & T_{2,2}\end{bmatrix}\begin{bmatrix}U_1 & U_2\end{bmatrix}^*
\end{equation}
where $U_1\in\complex^{n\times r}$ and $U_2\in\complex^{n-r\times r}$ have orthonormal columns, and $T_{1,1}$ and $T_{2,2}$ are upper triangular. 

Adding one additional assumption about the norm of $T_{1,2}$ our results extend via Theorem~\ref{thm:schur} to the Schur form. For simplicity we have been loose with our assumptions on $\|E\|_2$ and $\|T_{1,2}\|_2$ and small improvements to the necessary constants are possible. Importantly, rather then $T_{1,2}$ showing up in the upper bounds it shows up in the assumptions\emthin thereby controlling the matrices for which this result is valid.
\begin{theorem}
\label{thm:schur}
Let $A\in\reals^{n\times n}$ have the Schur form
\[
A = \begin{bmatrix}U_1 & U_2\end{bmatrix}\begin{bmatrix}T_{1,1} & T_{1,2} \\ 0 & T_{2,2}\end{bmatrix}\begin{bmatrix}U_1 & U_2\end{bmatrix}^*
\]
following~\eqref{eqn:SchurForm} and let 
 $\gap = \min\{\sep_2(T_{1,1},T_{2,2}),\sep_{(2,\infty),U_2}(T_{1,1},U_2T_{2,2}U_2^*)\}.$ If \\$\|E\|_2 \le \frac{\gap}{10},$ and $\|T_{1,2}\|_2 \le \frac{\gap}{10}$ then there exists a matrix $\hY\in\complex^{n-r\times r}$ such that
\[
\hU_1 = (U_1 + U_2\hY)(I+\hY^*\hY)^{-1/2}
\]
forms an invariant subspace for $\hA$ satisfying
\begin{align*}
\min\{\normtwoinf{\hU_1Q - U_1}:Q\in\orth^r\} &\le 8\normtwoinf{U_1}\left(\frac{\|E\|_2}{\sep_2(T_{1,1},T_{2,2})}\right)^2\\&\phantom{\le}+ 2\frac{\normtwoinf{U_2U_2^*EU_1}}{\gap} + 4\frac{\normtwoinf{U_2U_2^*EU_2U_2^*}\|E_{2,1}\|_2}{\gap\times \sep_2(T_{1,1},T_{2,2})}. \nonumber
\end{align*}
\end{theorem}
\begin{proof}
Following the proof of the main result, in this setting the upper right block of~\eqref{eqn:AhatDiag} is no longer zero, but $\hV_1^T\hA\hV_2.$ However, our additional assumptions ensure that $\|\hA_{1,2}\|_2\leq \gap / 5.$ Therefore, the result follows from the same argument as Lemma~\ref{lem:NewtonKantorovichX} and Theorem~\ref{thm:main} where we simply use $\hA_{1,2}$ in place of $E_{1,2}$ and $T_{1,1}$ and $T_{2,2}$ in lieu of $\Lambda_1$ and $\Lambda_2.$ 
\end{proof}
\begin{remark}
While we have formulated Theorem~\ref{thm:schur} for general $A$ and $E,$ in the case where $A$ is normal the upper bound simplifies significantly (independent of any structural assumptions on $E$). In particular, $T_{1,1}$ and $T_{2,2}$ become diagonal and $T_{1,2}=0.$ This implies that when we consider invariant subspaces of $\hA$ for symmetric $A$ we can constructed qualitatively similar bounds regardless of whether or not $E$ is symmetric.
\end{remark}

\subsection{Singular vectors and subspaces}
More generally, and perhaps of more interest for non-normal matrices, analogous questions about subspace perturbations can be posed for singular subspaces. While omitted here, we believe the proof strategy employed in this work can be extended to develop similar bounds for pairs of singular subspaces. This assertion is based on the quadratic forms given in~\cite{stewart1973Error} for singular subspaces of $A+E,$ though we leave such developments for future work.

\section{Conclusions}
Throughout this manuscript we have developed bounds on $\normtwoinf{\hV_1-V_1\tU}$ that are characterized by easily interpretable quantities (such as $\normtwoinf{V_1}$) and rely on minimal assumptions. By additionally demonstrating that various aspects of our bounds are ``essential'' when allowing for arbitrary symmetric $A$ and $E$ we clearly show where the limits are for this problem absent additional assumptions. Nevertheless, this effort also provides a natural launching point for further analysis, as it points to the key assumptions that have to (or may) be made to further understand the behavior of $\normtwoinf{\hV_1-V_1\tU}$ in specific settings. One concrete example of this is the random setting explored in Section~\ref{sec:numerics}, where more refined control of $\normtwoinf{E\hY}$ is possible. Lastly, there are several ways in which our bounds show commonly made assumptions in prior work (such as incoherence of $A$ or certain assumptions on $V_2\Lambda_2V_2^T$) are unnecessary. The consequence of this is that our bounds are sharper in certain situations. Collectively, we believe that these qualities make our bounds useful and interpretable across a broad range of applications. 

\section*{Acknowledgments}
We would like to thank the anonymous referees for their many helpful suggestions that improved this manuscript.

\appendix
\section{Proofs on properties of separation}
\subsection{Proof of Lemma~\ref{lem:sep_diag}}
\label{sec:proofSepDiag}
First, we observe that because $\sep$ is shift invariant it suffices to prove the result for non-negative $D_1$ and $D_2.$ Therefore, we assume $D_1$ and $D_2$ have non-negative entries for the remainder of this proof. We now prove lower bounds for all three variants of $\sep.$

{\sc 2-norm.} For any $Z$ we let $U_Z\Sigma_ZV_Z^T$ denote its reduced SVD and note that since $\|Z\|_2=1$ we have that $\sigma_1 = 1.$ Now we observe that
\begin{align*}
\|ZD_1 - D_2Z\|_2 &\geq \|ZD_1\|_2 - \|D_2Z\|_2 \\
&\geq \lambda_{\min}(D_1) - \|D_2\|_2 \\
&\geq \lambda_{\min}(D_1) - \lambda_{\max}(D_2)
\end{align*}
where we have used that 
\begin{align*}
\|XD_1\|_2 &= \|\Sigma_Z V_Z^T D_1\|_2 \\
&\geq \|\Sigma_Z V_Z^T D_1 V_Ze_1\|_2 \\
&\geq \|\sigma_1 e_1^T V_Z^T D_1 V_Ze_1\|_2\\
&\geq \lambda_{\min}(D_1).
\end{align*}

{\sc Frobenius norm.}
For any $Z$ with unit Frobenius norm observe that 
\begin{align*}
\|ZD_1 - D_2Z\|_F &\geq \|ZD_1\|_F - \|D_2Z\|_F \\
&\geq \lambda_{\min}(D_1) - \|D_2Z\|_2 \\
&\geq \lambda_{\min}(D_1) - \lambda_{\max}(D_2).
\end{align*}

{\sc 2,$\infty$-norm.}
For any $Z$ with $\normtwoinf{Z} = 1$ there exists an index $k$ such that $\|e_k^TZ\|_2 = 1$ where $e_k\in\reals^{m}$ is a canonical basis vector. Now, observe that 
\begin{align*}
\normtwoinf{ZD_1 - D_2Z} &\geq \normtwoinf{ZD_1} - \normtwoinf{D_2Z}\\
&\geq \|e_k^TZD_1\|_2 - \normtwoinf{D_2Z}\\
&\geq \lambda_{\min}(D_1) - \lambda_{\max}(D_2),
\end{align*}
where the last inequality follows because $D_2$ represents a row scaling of $Z.$

Finally, let $j$ denote the column in which $\lambda_{\min}(D_1)$ arises and let $i$ denote the column in which $\lambda_{\max}(D_2)$ arises. Now, observe that 
\[
e_ie_j^TD_1 - D_2e_ie_j^T = \left(\lambda_{\min}(D_1) - \lambda_{\max}(D_1)\right)e_ie_j^T,
\]
where $e_i\in\reals^{m}$ and $e_j\in\reals^{\ell}$ are canonical basis vectors.
Since 
\[
\|e_ie_j^T\|_2 = \|e_ie_j^T\|_F = \normtwoinf{e_ie_j^T} = 1
\]
we achieve the aforementioned lower bounds in all cases and thereby conclude the proof.

\section{The Newton-Kantorovich theorem}

This is the version of the the Newton-Kantorovich theorem we appeal to. It appears in \cite[pp 536]{kantorovich1982Functional}.

\begin{theorem}
\label{thm:NewtonKantorovich}
Let $X,Y$ be Banach spaces and $F:X\to Y$ be twice-continuously (Frechet) differentiable in a neighborhood of $U$ of $x\in X$. Assume there is a linear map $J:X\to Y$ such that $S_A^{-1}$ is bounded and satisfies
\begin{enumerate}
\item $\|J^{-1}(F(x))\| \le \eta$,
\item $\|J^{-1}\circ \partial F(x) - I\| \le \delta$,
\item $\|J^{-1}\circ \partial^2F(y)\| \le K$ for all $y\in U$.
\end{enumerate}
If $\delta < 1$ and $h := \frac{\eta K}{(1 - \delta)^2} < \frac12$, then the sequence $(x_t)$ defined recursively as
\[
\begin{aligned}
x_0 &\gets x \\
x_{t+1} &\gets x_t - J^{-1}(F(x_t))
\end{aligned}
\]
converges to 
$\bar{x}\in X$ such that $F(\bar{x}) = 0$ and 
\[
\|\bar{x} - x\| \le \frac{2\eta}{(1-\delta)(1 + \sqrt{1 - 2h})}.
\]
\end{theorem}

\subsection{Proof of Lemma \ref{lem:NewtonKantorovichX}}
\label{secA:NKX}
We start by evaluating the derivatives of $F$:
\[
\begin{aligned}
\partial F(0):X \to X\hA_{1,1} - \hA_{2,2}X, \\
\partial^2F(X):X_1,X_2 \to X_1\hA_{1,2}X_2.
\end{aligned}
\]
Recognizing $\|S_{\hA}^{-1}\|_2 = \frac{1}{\sep_2(\hA_{1,1},\hA_{2,2})}$ and $F(0) = E_{2,1}$, we have
\[
\begin{aligned}
\|S_{\hA}^{-1}(F(0))\|_2 &\le \frac{\|E_{2,1}\|_2}{\sep_2(\hA_{1,1},\hA_{2,2})},
\end{aligned}
\]
so the first condition of Theorem \ref{thm:NewtonKantorovich} is satisfied by $\eta = \frac{\|E_{2,1}\|_2}{\sep_2(\hA_{1,1},\hA_{2,2})}$. We recognize $S_{\hA} = \partial F(0)$, so the second condition of Theorem \ref{thm:NewtonKantorovich} is satisfied by $\delta = 0$. Finally, we have
\[
\begin{aligned}
\|S_{\hA}^{-1}(\partial^2F(0)(X_1,X_2))\|_2 &\le \frac{\|X_1\hA_{1,2}X_2\|_2}{\sep_2(\hA_{1,1},\hA_{2,2})} \le \frac{\|X_1\|_2\|E_{1,2}\|_2\|X_2\|_2}{\sep_2(\hA_{1,1},\hA_{2,2})},
\end{aligned}
\]
so the third condition of Theorem \ref{thm:NewtonKantorovich} is satisfied by $K = \frac{\|E_{1,2}\|_2}{\sep_2(\hA_{1,1},\hA_{2,2})}$. From Proposition 2.1 of \cite{karow2014Perturbation} we then have that
\[
\begin{aligned}
\sep_2(\hA_{1,1},\hA_{2,2}) &\ge \sep_2(\Lambda_1 + E_{1,1},\Lambda_2 + E_{2,2}) \\
&= \sep_2(\Lambda_1,\Lambda_2) - \|E_{1,1}\|_2 - \|E_{2,2}\|_2 \\
&\ge \sep_2(\Lambda_1,\Lambda_2)- 2\|E\|_2.
\end{aligned}
\]
We combine this bound on $\sep_2(\hA_{1,1},\hA_{2,2})$ with the condition $\|E\|_2 \le \frac{\sep_2(\Lambda_1,\Lambda_2)}{4}$ to obtain
\[
\begin{aligned}
\eta = \frac{\|E_{2,1}\|_2}{\sep_2(\hA_{1,1},\hA_{2,2})} \le \frac{\|E_{2,1}\|_2}{\sep_2(\Lambda_1,\Lambda_2) - 2\|E\|_2} \le \frac{2\|E_{2,1}\|_2}{\sep_2(\Lambda_1,\Lambda_2)}, \\
h = \frac{\eta K}{(1 - \delta)^2} = \frac{\|E_{2,1}\|_2\|E_{1,2}\|_2}{\sep_2(\hA_{1,1},\hA_{2,2})^2} \le \frac{\|E\|_2^2}{(\sep_2(\Lambda_1,\Lambda_2) - 2\|E\|_2)^2} \le \frac14 < \frac12,
\end{aligned}
\]
so the NK theorem implies $F$ has a root $\hX$ such that
\[
\|\hX\|_2 \le \frac{2\eta}{(1-\delta)(1 + \sqrt{1 - 2h})} < 2\eta \le \frac{4\|E_{2,1}\|_2}{\sep_2(\Lambda_1,\Lambda_2)}
\]
as claimed.

\section{Multiple perturbations}
Here, we characterize how changes in the invariant subspaces may be controlled when we have multiple perturbations of the matrix $A$. Our primary use of this result is in developing Theorem~\ref{thm:probG}. Specifically, when $E$ is drawn from a random model satisfying certain independence assumptions we can use multiple draws of $E$ to develop tighter bounds on the changes in $V_1$ generated by any single instance of $E.$

\begin{lemma}
\label{lem:prob}
Let $A\in\reals^{n\times n}$ be symmetric with eigen-decomposition 
\[
A = V_1\Lambda_1V_1^T +V_2\Lambda_2V_2^T
\] 
following our conventions in~\eqref{eqn:eigen-decomp} and  
\[
\gap = \min\{\sep_2(\Lambda_1,\Lambda_2),\sep_{(2,\infty),V_2}(\Lambda_1,V_2\Lambda_2V_2^T)\}.
\] 
Consider two symmetric perturbations $E$ and $\tE$ satisfying $E_{1,2} = \tE_{1,2},$ $E_{1,1} = \tE_{1,1},$ $\|E\|_2\leq\frac{\gap}{5},$ and $\|\tE\|_2\leq\frac{\gap}{5}.$ In this setting
\begin{equation*}
\begin{aligned}
\normtwoinf{\hV_1\tU - V_1} &\le 8\normtwoinf{V_1}\left(\frac{\|E_{2,1}\|_2}{\sep_2(\Lambda_1,\Lambda_2)}\right)^2\\&\phantom{\le}+ 2\frac{\normtwoinf{V_2E_{2,1}}}{\gap} + 4\frac{\normtwoinf{V_2E_{2,2}V_2^T\tY}}{\gap} \\&\phantom{=} + 5\frac{(\|E_{2,2}\|_2+\|\tE_{2,2}\|_2)\normtwoinf{V_2E_{2,2}V_2^T}\|E_{2,1}\|_2}{\gap\times\sep_{2}(\Lambda_1,\Lambda_2)^2}\\&\phantom{=} + 10\frac{\normtwoinf{V_2E_{2,2}V_2^T}\|E_{2,1}\|_2^3}{\gap\times\sep_{2}(\Lambda_1,\Lambda_2)^3},
\end{aligned}
\end{equation*}
where $\hV_1$ is any matrix with orthonormal columns whose range is the dominant $r$-dimensional invariant subspace of $A+E,$ $\tU$ solves the orthogonal Procrustes problem
\begin{equation*}
\min\{\|\hV_1U - V_1\|_F:U\in\orth^r\}.
\end{equation*}
In addition, $\tY\in\reals^{n\times r}$ is in the range of $V_2,$ is a root of
\[
G:Y\to -V_2\tA_{2,1} + Y\tA_{1,1} - V_2\tA_{2,2}V_2^TY + Y\tA_{1,2}V_2^TY
\]
where $\tA = A+\tE,$ and satisfies $\|\tY\|_2\leq \frac{4 \|E_{2,1}\|_2}{\sep_2(\Lambda_1,\Lambda_2)}.$
\end{lemma}
\begin{proof}
First, the assumption $\|E\|_2 \leq \frac{\gap}{5}$ allows us to invoke Lemma~\ref{lem:Y} and deduce the first two terms in the bound directly. To construct the last three terms we bound
\[
4\frac{\normtwoinf{V_2E_{2,2}V_2^T\hY}}{\gap}
\]
in a more nuanced manner than the na\"ive sub-multiplicative bound used in Section~\ref{sec:main}. 

Letting $\tA = A+\tE$ we may invoke Lemma~\ref{lem:NewtonKantorovichX} to assert that there exists $\tX\in\reals^{(n-r)\times r}$ such that
\[
\tV_1 = (V_1+V_2\tX)(I_r+\tX^T\tX)^{-\frac12}
\]
is an orthonormal basis for the dominant $r$-dimensional invariant subspace of $\tA$ and $\|\tX\|_2\leq \frac{4 \|E_{2,1}\|_2}{\sep_2(\Lambda_1,\Lambda_2)}.$ As before, $\tX$ is a root of 
\[
\begin{aligned}
F:X\to -\tA_{2,1} + X\tA_{1,1} - \tA_{2,2}X + X\tA_{1,2}X
\end{aligned}
\]
and therefore $\tY = V_2\tX$ is a root of
\[
G:Y\to -V_2\tA_{2,1} + Y\tA_{1,1} - V_2\tA_{2,2}V_2^TY + Y\tA_{1,2}V_2^TY.
\]

We now proceed by observing that
\[
\begin{aligned}
4\frac{\normtwoinf{V_2E_{2,2}V_2^T\hY}}{\gap} &\leq 4\frac{\normtwoinf{V_2E_{2,2}V_2^T\tY}}{\gap} + 4\frac{\normtwoinf{V_2E_{2,2}V_2^T(\hY-\tY)}}{\gap}\\
&\leq 4\frac{\normtwoinf{V_2E_{2,2}V_2^T\tY}}{\gap} + 4\frac{\normtwoinf{V_2E_{2,2}V_2^T}\|\hY-\tY\|_2}{\gap}.
\end{aligned}
\]

Control over $\|\hY-\tY\|_2$ is achieved by observing that $\tY$ and $\hY$ satisfy
\[
0 = -V_2\hA_{2,1} + \hY\hA_{1,1} - V_2\hA_{2,2}V_2^T\hY + \hY\hA_{1,2}V_2^T\hY
\]
and
\[
0 = -V_2\tA_{2,1} + \tY\tA_{1,1} - V_2\tA_{2,2}V_2^T\tY + \tY\tA_{1,2}V_2^T\tY.
\]
In particular, since $\tA_{i,j} = \hA_{i,j}$ for all $i,j$ pairs except $i=j=2$ subtracting the equations yields 
\[
0 = (\hY-\tY)\hA_{1,1} - V_2\hA_{2,2}V_2^T\hY + V_2\tA_{2,2}V_2^T\tY + \hY\hA_{1,2}V_2^T\hY - \tY\tA_{1,2}V_2^T\tY,
\]
and via further rearrangement
\[
\begin{aligned}
(\hY-\tY)\hA_{1,1} - V_2A_{2,2}V_2^T(\hY-\tY) &= V_2E_{2,2}V_2^T\hY - V_2\tE_{2,2}V_2^T\tY \\
&\phantom{=} - \hY\hA_{1,2}V_2^T\hY + \tY\tA_{1,2}V_2^T\tY.
\end{aligned}
\]
Taking norms we observe that the left hand side is at least
\begin{equation}
\frac{4}{5}\sep_{2}(\Lambda_1,\Lambda_2)\|(\hY-\tY)\|_2 \leq \|(\hY-\tY)\hA_{1,1} - V_2A_{2,2}V_2^T(\hY-\tY)\|_2
\label{eqn:probLHS}
\end{equation}
where we have used that $\|E_{1,1}\|_2\leq \sep_{2}(\Lambda_1,\Lambda_2)/5.$ We may bound the right hand side via the triangle inequality and repeated used of Lemma~\ref{lem:NewtonKantorovichX} as
\begin{equation}
\begin{aligned}
\|V_2E_{2,2}V_2^T\hY - V_2\tE_{2,2}V_2^T\tY &- \hY\hA_{1,2}V_2^T\hY + \tY\tA_{1,2}V_2^T\tY\|_2 \leq \\&4\frac{(\|E_{2,2}\|_2+\|\tE_{2,2}\|_2)\|E_{2,1}\|}{\sep_{2}(\Lambda_1,\Lambda_2)} + 8\frac{\|E_{2,1}\|_2^3}{\sep_{2}(\Lambda_1,\Lambda_2)^2},
\end{aligned}
\label{eqn:probRHS}
\end{equation}
where we have also used that $A_{1,2} = 0.$ Combining~\eqref{eqn:probLHS} and~\ref{eqn:probRHS} yields the desired upper bound.
\end{proof}

\bibliographystyle{siam}
\bibliography{YK}

\end{document}